%% file: paper.tex
\documentclass{IEEEtran}

\IEEEoverridecommandlockouts                              % This command is only needed if 
\usepackage{mathrsfs}
\usepackage[font={small}]{caption}
\usepackage{scalerel}
\usepackage{url}
\usepackage{bbold}
\usepackage{amsfonts}
\usepackage{amsmath,amssymb,amsfonts}
\usepackage{amsthm}
\usepackage{graphicx}
\usepackage{caption}  %For the usage of subfigures
\usepackage{subcaption} %For the usage of subfigures
\usepackage{wrapfig}
\usepackage{mathrsfs} 
\usepackage{algorithm,algorithmic}
\usepackage{array}
\usepackage{times}
\usepackage{url}
\usepackage{cite}
\newcommand{\citet}[1]{\cite{#1}}
\usepackage{upgreek}
\usepackage{float}
\usepackage{longtable}
\usepackage{color}
\usepackage{wasysym}
\usepackage{grffile}
\usepackage{stackengine}

\usepackage{mathtools, nccmath}
\usepackage{scalerel}

\allowdisplaybreaks

\input{header.tex}

\title{\LARGE \bf Regret Analysis of Distributed Online LQR Control \\ for Unknown LTI Systems}

\author{Ting-Jui Chang and Shahin Shahrampour, {\it Senior Member}, {\it IEEE}  
\thanks{T.J. Chang and S. Shahrampour are with the Department of Mechanical and Industrial Engineering, Northeastern University, Boston, MA 02115, USA. 
        {\tt\small email:\{chang.tin,s.shahrampour\}@northeastern.edu}.}%
        \thanks{This work is supported in part by NSF ECCS-2136206 Award.}
}

\begin{document}

\maketitle
\thispagestyle{empty}
\pagestyle{empty}

%%%%%%%%%%%%%%%%%%%%%%%%%%%%%%%%%%%%%%%%%%%%%%%%%%%%%%%%%%%%%%%%%%%%%%%%%%%%%%%%
\begin{abstract}
Online optimization has recently opened avenues to study optimal control for  time-varying cost functions that are unknown in advance. Inspired by this line of research, we study the distributed online linear quadratic regulator (LQR) problem for linear time-invariant (LTI) systems with unknown dynamics. Consider a multi-agent network where each agent is modeled as a LTI system. The network has a global {\it time-varying} quadratic cost, which may evolve adversarially and is only {\it partially} observed by each agent sequentially. The goal of the network is to collectively (i) estimate the unknown dynamics and (ii) compute local control sequences competitive to the best centralized policy in hindsight, which minimizes the sum of network costs over time. This problem is formulated as a {\it regret} minimization. We propose a distributed variant of the online LQR algorithm, where agents compute their system estimates during an exploration stage. Each agent then applies distributed online gradient descent on a semi-definite programming (SDP) whose feasible set is based on the agent system estimate. We prove that with high probability the regret bound of our proposed algorithm scales as $O(T^{2/3}\log T)$, implying the consensus of all agents over time. We also provide simulation results verifying our theoretical guarantee.
\end{abstract}

%%%%%%%%%%%%%%%%%%%%%%%%%%%%%%%%%%%%%%%%%%%%%%%%%%%%%%%%%%%%%%%%%%%%%%%%%%%%%%%%
\section{Introduction}
In recent years, there has been a significant interest on problems arising at the interface of control and machine learning. Among classical control problems, LQR control \cite{4309169, bertsekas1995dynamic, zhou1996robust} is a prominent point in case. LQR control centers around LTI systems, where the control-state pairs introduce a quadratic cost with {\it time-invariant} parameters. When the dynamics of the LTI system is known, for finite-horizon and infinite-horizon problems, the optimal controllers have closed-form solutions, which can be derived by solving the corresponding Riccati equations. 

Despite the excellent insights on the LQR problem provided by the classical control theory, in practical problems we might encounter two challenges. (I) The environment could change in an unpredictable way, which makes the cost parameters {\it time-varying} and {\it unknown} in advance (e.g., in variable-supply electricity production and building climate control with time-varying energy costs \citet{cohen2018online}). (II) Furthermore, the dynamics of the LTI system may be {\it unknown}. The former challenge has motivated research at the interface of online optimization and control, where online LQR problem is cast as a {\it regret} minimization and the performance of an online algorithm is compared to that of the best fixed control policy in hindsight. The regret metric is particularly meaningful in the online setting, where the cost parameters are unknown in advance. The focus of online LQR is on the finite-time performance from a learning-theory perspective (see details of this literature in item 4 of Subsection \ref{related}). The latter challenge is addressed via adaptive control in general. In this case, the learner must strike a balance between exploration (estimating the system dynamics while preventing the cumulative cost from going unbounded) and exploitation (using the estimates to compete with the performance of the optimal controller) \cite{abbasi2011regret,ibrahimi2012efficient,dean2018regret,cohen2019learning}. 

In this work, we consider the distributed online LQR problem for a network of LTI systems with {\it unknown} dynamics. Each system is represented by an agent in the network that has a global {\it time-varying} quadratic cost. The cost sequence may evolve adversarially and is only {\it partially} observed by each agent sequentially. The goal of each agent is to generate a control sequence that is competitive to that of the best centralized policy in hindsight, formulated by {\it regret}. This setting can be applied for modeling the energy consumption in mobile sensor networks as described in Example \ref{example}. To address the problem, we propose a decentralized algorithm with two phases. In the exploration phase, each agent computes system estimates using the {\bf EXTRA} algorithm \cite{shi2015extra}, which is an iterative decentralized optimization method. In the exploitation phase, agents perform distributed online gradient descent on a SDP (whose feasible set is constructed by local system estimates) and extract the control policies accordingly. We prove that if every agent maintains a good balance between system identification (exploration) and online control (exploitation), the regret is bounded by $O(T^{2/3}\log T)$, where $T$ is the total number of iterations. This implies that the agents reach consensus and collectively compete with the best fixed controller in hindsight. Besides the exploration-exploitation trade-off, the main technical challenge is that the decentralized system identification step results in different SDPs across agents. This implies that the feasible set of SDP varies from one agent to another, and we cannot directly use distributed online optimization results on a common feasible set. We draw upon techniques from alternating projections to tackle this problem. Our technical proof is provided in the Appendix (Section \ref{appendix}). We also provide simulation results verifying our regret bound.

\subsection{Related Literature}\label{related}
%Our work is at the interface of several strands of literature.
\noindent
\textbf{(1) Distributed LQR Control:}
Distributed LQR has been widely studied in the control literature. A number of works focus on multi-agent systems with known, identical decoupled dynamics. In \cite{4626964}, a distributed control design is proposed by solving a single LQR problem whose size scales with the maximum degree of the graph capturing the network. The authors of \citet{6862471} derive the necessary condition for an optimal distributed controller design, resulting in a non-convex optimization problem. The work of \cite{5299181} addresses a multi-agent network, where the dynamics of each agent is a single integrator. The authors of \cite{5299181} show that the computation of the optimal controller requires the knowledge of the graph and the initial information of all agents. Given the difficulty of precisely solving the optimal distributed controller, Jiao et al. \citet{8736845} provide the sufficient conditions to obtain sub-optimal controllers. All of the aforementioned works need global information such as network topology to compute the controllers. On the other hand, Jiao et al. \citet{8734804} propose a decentralized method to compute the controllers and show that the system will reach consensus. For the case of unknown dynamics, Alemzadeh et al. \cite{alemzadeh2019distributed} propose a distributed Q-learning algorithm for dynamically decoupled systems. There are other works focusing on distributed control without assuming identical decoupled sub-systems. Fattahi et al. \cite{fattahi2019efficient} study distributed controllers for unknown and sparse LTI systems. Furieri et al. \cite{furieri2020learning} address model-free methods for distributed LQR problems and provide sample-complexity bounds for problems with local gradient dominance property (e.g., quadratically-invariant problems). The work of \cite{furieri2020first} investigates the convergence of distributed controllers to a global minimum for quadratically invariant problems with first-order methods. %They also characterize another class of uniquely stationary problems which possess the same property.}

\vspace{0.2cm}
\noindent
{\bf (2) System Identification of LTI Systems:} For solving LQR problems with unknown dynamics, we first need to learn the underlying system. To provide performance guarantee for the controller, it is important to explicitly quantify the uncertainty of the model estimate. The classical theory of system identification for LTI systems (e.g., \cite{aastrom1971system,ljung1999system,chen2012identification,goodwin1977dynamic}) characterizes the asymptotic properties of the estimators. On the contrary, recent results in statistical learning focus on finite-time guarantees. In \cite{dean2019sample}, it is shown that for fully observable systems, a least-squares estimator can learn the underlying dynamics from multiple trajectories. % with length scaling linearly to the system dimension. 
These results are later extended to the estimation using a single trajectory \cite{simchowitz2018learning, sarkar2019near}. For partially observable systems, estimators with polynomial sample complexities are provided in the literature (e.g., \cite{oymak2019non, sarkar2019finite, tsiamis2019finite, simchowitz2019learning}), and the work of \cite{fattahi2020learning} improves the sample complexity to be poly-logarithmic.

\vspace{0.2cm}
\noindent
{\bf (3) Online LQR with Unknown Dynamics and Time-Invariant Costs:} There is a recent line of research dealing with LQR control problems with unknown dynamics. Several techniques are proposed using (i) gradient estimation (e.g., \cite{fazel2018global, malik2019derivative, 9130755, 9147749}), (ii) the estimation of dynamics matrices and derivation of the controller by considering the estimation uncertainty (e.g., \cite{dean2018regret, dean2019sample, cohen2019learning, cassel2020logarithmic, simchowitz2020naive, lale2020explore}), and (iii) wave-filtering \cite{hazan2017learning, arora2018towards}.

\vspace{0.2cm}
\noindent
\textbf{(4) Online Control with Time-Varying Costs:}
Recently, there has been a significant interest in studying linear dynamical systems with time-varying cost functions, where online learning techniques are applied. This literature investigates two scenarios: \textbf{I) Known Systems:} Cohen et al. \citet{cohen2018online} study the SDP relaxation for online LQR control and establish a regret bound of $O(\sqrt{T})$ for known LTI systems with time-varying quadratic costs. Agarwal et al. \citet{agarwal2019online} propose the disturbance-action policy parameterization and reduce the online control problem to online convex optimization with memory. They show that for adversarial disturbances and arbitrary  time-varying convex functions, the regret is $O(\sqrt{T})$. Agarwal et al. \cite{agarwal2019logarithmic} consider the case of time-varying strongly-convex functions and improve the regret bound to $O(\text{poly}(\text{log}T))$. Simchowitz et al. \citet{simchowitz2020improper} further extend the $O(\text{poly}(\text{log}T))$ regret bound to partially observable systems with semi-adversarial disturbances. Yu et al. \cite{yu2020power} incorporate the idea of model predictive control into online LQR control with time-invariant cost function and correct noise predictions. Zhang et al. \cite{zhang2021regret} extend this idea to the setup where costs are time-varying and accurate disturbance predictions are not accessible. Both of them provide the dynamic regret bound with a term shrinking exponentially with the prediction window. Our previous work \cite{chang2021distributed} studies the distributed online LQR control with known dynamics and provides the regret bound of $O(\sqrt{T})$. \textbf{II) Unknown Systems:} For fully observable systems, Hazan et al. \citet{hazan2020nonstochastic} derive the regret of $O(T^{2/3})$ for time-varying convex functions with adversarial noises. For partially observable systems, the work of \cite{simchowitz2020improper} addresses the cases of (i) convex functions with adversarial noises and (ii) strongly-convex functions with semi-adversarial noises, and provide regret bounds of $O(T^{2/3})$ and $O(\sqrt{T})$, respectively. Lale et al. \citet{lale2020logarithmic} establish an $O(\text{poly}(\text{log}T))$ regret bound for the case of stochastic perturbations, time-varying strongly-convex functions, and partially observed states.

Our work lies precisely at the interface of distributed LQR, online LQR and adaptive control, addressing distributed online LQR with unknown dynamics.  

\section{Preliminaries and Problem Formulation}
\subsection{Notation} 
%We use the following notation in this work:

{\small
\begin{tabular}{|c||l|}
    \hline
    $[n]$ & The set $\{1,2,\ldots,n\}$ for any integer $n$ \\
    \hline
    $\text{Tr}(\cdot)$ & The trace operator\\
    \hline
    $\norm{\cdot}$ & Euclidean (spectral) norm of a vector (matrix)\\
    \hline
    $\norm{\cdot}_F$ & Frobenius norm of a matrix\\
    \hline
    $\mathrm{E}[\cdot]$ & The expectation operator\\
    \hline
    $\Pi_{\Sc}[\cdot]$ & The operator for the projection to set $\Sc$ \\
    \hline
    $[\Ab]_{ij}$ & The entry in the $i$-th row and $j$-th column of $\Ab$\\
    \hline
    $[\Ab]_{:,j}$ & The $j$-th column of $\Ab$\\
    \hline
    \rule{0pt}{9pt}
    $\Ab\bullet\Bb$ & $\text{Tr}(\Ab^\top\Bb)$\\
    \hline
    $\Ab\succeq\Bb$ & $(\Ab-\Bb)$ is positive semi-definite\\
    \hline
    $\1$ & The vector of all ones\\
    \hline
    $\eb_i$ & The $i$-th basis vector\\
    \hline
    $\text{vec}(\Ab)$ & Vectorized version of the matrix $\Ab$\\
    \hline
\end{tabular}}

\subsection{Distributed Online LQR Control with Unknown Dynamics}
We consider a multi-agent network of $m$ LTI systems, where the dynamics of agent $i$ is given as, 
$$\xb_{i,t+1} = \Ab\xb_{i,t}+\Bb\ub_{i,t}+\wb_{i,t},\quad i\in [m]$$
and $\xb_{i,t}\in\mathrm{R}^d$ and $\ub_{i,t}\in\mathrm{R}^k$ represent agent $i$ state  and control (or action) at time $t$, respectively. Furthermore, $\Ab\in\mathrm{R}^{d\times d}$, $\Bb\in\mathrm{R}^{d\times k}$, and $\wb_{i,t}$ is a Gaussian noise with zero mean and covariance $\Wb \succeq \sigma^2\Ib$. The system parameters $(\Ab,\Bb)$ are {\it unknown} to all agents and need to be estimated. The noise sequence $\{\wb_{i,t}\}$ is independent over time and agents. We also assume that $\norm{[\Ab\:\Bb]}_F\leq \vartheta$ and let $n:=d+k$ for the presentation simplicity.

Departing from the classical LQR control, we consider the {\it online} distributed LQR problem, where the cost functions are {\it unknown} in advance. At round $t$, agent $i$ receives the state $\xb_{i,t}$ and applies the action $\ub_{i,t}$. Then, positive semi-definite cost matrices $\Qb_{i,t}$ and $\Rb_{i,t}$ are revealed, and the agent incurs the cost $\xb_{i,t}^\top\Qb_{i,t}\xb_{i,t} + \ub_{i,t}^\top\Rb_{i,t}\ub_{i,t}$. Throughout this paper, we assume that $\text{Tr}(\Qb_{i,t}),\text{Tr}(\Rb_{i,t})\leq C$ for all $i,t$ and some $C>0$. Agent $i$ follows a policy that selects the control $\ub_{i,t}$ based on the observed cost matrices $\Qb_{i,1},\ldots,\Qb_{i,t-1}$ and $\Rb_{i,1},\ldots,\Rb_{i,t-1}$, as well as the information received from its {\it local} neighborhood. This policy is not driven based on individual costs. On the contrary, agents follow a {\it team} goal through minimizing a cost collectively as we describe next.

\vspace{.2cm}
\noindent
{\bf Centralized Benchmark:} In order to gauge the performance of a distributed online LQR algorithm, we require a centralized benchmark. In this paper, we focus on the {\it finite-horizon} problem, where for a centralized policy $\pi$, the cost after $T$ steps is given as 
\begin{align}\label{eq:cost-central}
    J_T(\pi)=\mathrm{E}\left[\sum_{t=1}^T {\xb_t^{\pi}}^\top\Qb_t\xb_t^{\pi} + {\ub_t^{\pi}}^\top\Rb_t\ub_t^{\pi}\right],
\end{align}
where $\Qb_t = \sum_{i=1}^m \Qb_{i,t}$ and $\Rb_t = \sum_{i=1}^m \Rb_{i,t}$, and the expectation is over the possible randomness of the policy as well as the noise. The superscript $\pi$ in $\ub_t^{\pi}$ and $\xb_t^{\pi}$ alludes that the state-control pairs are chosen by the policy $\pi$, given full access to cost matrices of all agents. Notice that in the {\it infinite-horizon} version of the problem with time-invariant cost matrices $(\Qb,\Rb)$, where the goal is to minimize $\lim_{T\rightarrow\infty}J_T(\pi)/T$, it is well-known that for a controllable LTI system $(\Ab,\Bb)$, the optimal policy is given by the constant linear feedback, i.e., $\ub_t^{\pi}=\Kb\xb_t^{\pi}$ for a matrix $\Kb\in \mathrm{R}^{k\times d}$.

\vspace{.2cm}
\noindent
{\bf Regret Definition:} The goal of a distributed online LQR algorithm $\mathcal{A}$ is to mimic the performance of an ideal centralized algorithm that solves \eqref{eq:cost-central}. The main two challenges are (i) the online nature of the problem, where cost matrices become available sequentially, and (ii) the distributed setup, where agent $i$ only receives information about the sequence $\{\Qb_{i,t},\Rb_{i,t}\}$ while the cost is based on $\{\Qb_{t},\Rb_{t}\}$. In this setting, each agent $j$ locally generates the control sequence $\{\ub_{j,t}\}_{t=1}^T$, that is competitive to the best policy among a benchmark policy class $\Pi$. This can be formulated as minimizing the individual regret, which is defined as follows
\begin{align}\label{eq:regret}
    \text{Regret}_T^j(\mathcal{A}):=J_T^j(\mathcal{A})-\min_{\pi\in\Pi}J_T(\pi),
\end{align}
for agent $j\in [m]$, where 
\begin{align}
    J_T^j(\mathcal{A}) %\mathrm{E}\left[\sum_{t=1}^T\sum_{i=1}^m {\xb_{j,t}^{\mathcal{A}}}^\top\Qb_{i,t}\xb_{j,t}^{\mathcal{A}} + {\ub_{j,t}^{\mathcal{A}}}^\top\Rb_{i,t}\ub_{j,t}^{\mathcal{A}}\right]\\
    &=\mathrm{E}\left[\sum_{t=1}^T {\xb_{j,t}^{\mathcal{A}}}^\top\Qb_{t}\xb_{j,t}^{\mathcal{A}} + {\ub_{j,t}^{\mathcal{A}}}^\top\Rb_{t}\ub_{j,t}^{\mathcal{A}}\right].
\end{align}
A successful distributed algorithm is one that keeps the regret sublinear with respect to $T$. Of course, this also depends on the choice of the benchmark policy class $\Pi$, which is assumed to be the set of strongly stable policies (to be defined precisely in Section \ref{sec:stability}). Since the underlying dynamics is unknown, agents have to find a good trade-off between {\it exploration} (estimating the system parameters) and {\it exploitation} (keeping the regret sublinear).

\vspace{.2cm}
\noindent
%The agents communicate locally to minimize the cost. The network topology is defined by a time-invariant doubly stochastic matrix $\Pb$, where $[\Pb]_{ji}>0$ if agent $j$ communicates with agent $i$; otherwise $[\Pb]_{ji}=0$. The network is assumed to be connected, and there exists a geometric mixing bound for $\Pb$ \cite{liu2008monte}, such that $\sum_{j=1}^m \left|[\Pb^k]_{ji}-1/m\right|\leq \sqrt{m}\beta^k,\:i\in [m],$ where $\beta$ is the second largest singular value of $\Pb$.
{\bf Network Structure:} %The agents communicate locally to minimize the cost. 
The underlying network topology is represented by a symmetric doubly stochastic matrix $\Pb$, i.e., all elements of $\Pb$ are non-negative and $\sum_{i=1}^m [\Pb]_{ji}=\sum_{j=1}^m [\Pb]_{ji}=1$. If $[\Pb]_{ji}>0$, agents $i$ and $j$ are neighbors; otherwise $[\Pb]_{ji}=0$. The network is assumed to be connected, i.e., for any two agents $i,j\in [m]$, there is a (potentially multi-hop) path from $i$ to $j$. We also assume $\Pb$ has a positive diagonal. Then, there exists a geometric mixing bound for $\Pb$ \cite{liu2008monte}, such that $\sum_{j=1}^m \left|[\Pb^k]_{ji}-1/m\right|\leq \sqrt{m}\beta^k,\:i\in [m],$
where $\beta$ is the second largest singular value of $\Pb$. Agents cannot share their observed cost functions with each other, but they can exchange a local parameter used for constructing the controllers. The communication is consistent with the structure of $\Pb$. We elaborate on this in the algorithm description.

\begin{example}\label{example} Our framework can be used for minimizing the energy consumption in mobile sensor networks (MSNs) \cite{guo2014cooperation} in time-varying settings. Consider a MSN where at time $t$ the total mobility cost (or budget) of sensors is modeled by matrices $(\Qb_t,\Rb_t)$. Each agent $i$ has a local budget of $(\Qb_{i,t},\Rb_{i,t})$, but the team goal is to design actions that minimize the global network cost over time. Then, actions of this MSN should be guided to minimize the global cost in \eqref{eq:cost-central}, though each sensor only has local information.
\end{example}

\subsection{Strong Stability and Sequential Strong Stability}\label{sec:stability}
We consider the set of strongly stable linear (i.e., $\ub=\Kb\xb$) controllers as the benchmark policy class. Following \cite{cohen2018online}, we define the notion of strong stability as follows. 
\begin{definition}\label{D: Strong Stability}(Strong Stability) A linear policy $\Kb$ is $(\kappa, \gamma)$-strongly stable (for $\kappa > 0$ and $0<\gamma\leq 1$) for the LTI system $(\Ab,\Bb)$, if $\norm{\Kb}\leq \kappa$, and there exist matrices $\Lb$ and $\Hb$ such that $\Ab+\Bb\Kb=\Hb\Lb\Hb^{-1}$, with $\norm{\Lb}\leq 1-\gamma$ and $\norm{\Hb}\|\Hb^{-1}\|\leq \kappa$.
\end{definition}
Strong stability is a quantitative version of stability, in the sense that any stable policy is strongly stable for some $\kappa$ and $\gamma$, and vice versa \cite{cohen2018online}. A strongly stable policy ensures fast mixing
and exponential convergence to a steady-state distribution. In particular, for the LTI system $\xb_{t+1}=\Ab\xb_t + \Bb\ub_t + \wb_t$, if a $(\kappa, \gamma)$-strongly stable policy $\Kb$ is applied ($\ub_t = \Kb\xb_t$), $\widehat{\Xb}_t$ (the state covariance matrix of $\xb_t$) converges to $\Xb$ (the steady-state covariance matrix) with the following exponential rate $$\norm{\widehat{\Xb}_t-\Xb}\leq \kappa^2 e^{-2\gamma t}\norm{\widehat{\Xb}_0-\Xb}.$$
See Lemma 3.2 in \cite{cohen2018online} for details. The {\it sequential} nature of {\it online} LQR control requires another notion of strong stability, called {\it sequential strong stability} \citet{cohen2018online}, defined as follows.

\begin{definition}\label{D: Sequential Strong Stability}(Sequential Strong Stability) A sequence of linear  policies $\{\Kb_t\}_{t=1}^T$ is $(\kappa,\gamma)$-strongly stable if there exist matrices $\{\Hb_t\}_{t=1}^T$ and $\{\Lb_t\}_{t=1}^T$ such that $\Ab+\Bb\Kb_t=\Hb_t\Lb_t\Hb_t^{-1}$ for all $t$ with the following properties,
\begin{enumerate}
    \item $\norm{\Lb_t}\leq 1-\gamma$ and $\norm{\Kb_t}\leq \kappa$.
    \item  $\norm{\Hb_t}\leq \beta^{\prime}$ and $\norm{\Hb_t^{-1}}\leq 1/\alpha^{\prime}$ with $\kappa=\beta^{\prime}/\alpha^{\prime}$.
    \item $\norm{\Hb_{t+1}^{-1}\Hb_t}\leq 1+\gamma/2$.
\end{enumerate}
\end{definition}
Sequential strong stability generalizes strong stability to the time-varying scenario, where a sequence of policies $\{\Kb_t\}_{t=1}^T$ is used. The convergence of steady-state covariance matrices induced by $\{\Kb_t\}_{t=1}^T$ is characterized as follows.
\begin{lemma}\label{L: Convergence of Sequential Strong Stability} (Lemma 3.5 in \cite{cohen2018online}) Suppose a time-varying policy ($\ub_t=\Kb_t\xb_t$) is applied. Denote the steady-state covariance matrix of $\Kb_t$ as $\Xb_t$. If $\{\Kb_t\}$ are $(\kappa,\gamma)$-sequentially strongly stable and $\norm{\Xb_t-\Xb_{t-1}}\leq \eta$, $\widehat{\Xb}_t$ (the state covariance matrix of $\xb_t$) converges to $\Xb_t$ as follows
\begin{equation*}
    \norm{\widehat{\Xb}_{t+1}-\Xb_{t+1}}\leq \kappa^2e^{-\gamma t}\norm{\widehat{\Xb}_1-\Xb_1} + \frac{2\eta\kappa^2}{\gamma}.
\end{equation*}
\end{lemma}

% \begin{lemma}\label{L: Convergence of Strong Stability} (Lemma 3.2 in \cite{cohen2018online}) For all $t\in \mathrm{N}$, let $\widehat{\Xb}_t$ be the state covariance matrix on round $t$ starting from some $\widehat{\Xb}_0$ and following a $(\kappa,\gamma)$-strongly stable policy $\pi(\xb)=\Kb\xb$. Then $\widehat{\Xb}_t$ approaches to a steady-state covariance matrix $\Xb$ such that 
% $$\norm{\widehat{\Xb}_t-\Xb}\leq \kappa^2 e^{-2\gamma t}\norm{\widehat{\Xb}_0-\Xb}.$$
% \end{lemma}

\iffalse
\begin{lemma}\label{L: Convergence of Sequential Strong Stability} (Lemma 3.5 in \cite{cohen2018online})
Define the policy on round $t$ as $\pi_t(\xb)=\Kb_t\xb$ ($t=1,2,\ldots$) where the sequence $\Kb_1,\Kb_2,\ldots$ is $(\kappa,\gamma)$-strongly stable with respective steady-state covariance matrices $\Xb_1,\Xb_2,\ldots$ If $\norm{\Xb_t-\Xb_{t-1}}\leq \eta$ for all $t$ and for some $\eta>0$, $\widehat{\Xb}_t$, the state covariance matrix on round $t$ starting from some $\widehat{\Xb}_1\succeq 0$, has the following property:
$$\norm{\widehat{\Xb}_{t+1}-\Xb_{t+1}}\leq \kappa^2 e^{-\gamma t}\norm{\widehat{\Xb}_1-\Xb_1}+\frac{2\eta\kappa^2}{\gamma}$$.
\end{lemma}
\fi

\subsection{SDP Relaxation for LQR Control}
For the following dynamical system
$$\xb_{t+1} = \Ab\xb_{t} + \Bb\ub_{t} + \wb_{t},\quad \wb_{t}\sim \Nc(0,\Wb),$$
the infinite-horizon version of \eqref{eq:cost-central}, i.e., $\minimize \lim_{T\rightarrow\infty}J_T(\pi)/T,$
with fixed cost matrices $\Qb$ and $\Rb$ can be relaxed via a SDP when the steady-state distribution exists. For $\nu>0$, the SDP relaxation is formulated as \cite{cohen2018online}
\begin{equation}\label{eq:SDP}
\begin{split}
    \text{minimize}\quad &J(\Sigma)=\begin{pmatrix}\Qb & 0\\0 & \Rb\end{pmatrix}\bullet\Sigma\\
    \text{subject to}\quad &\Sigma_{\xb\xb}=[\Ab\:\Bb]\Sigma[\Ab\:\Bb]^\top+\Wb,\\
    &\Sigma\succeq 0,\quad \text{Tr}(\Sigma)\leq\nu,
\end{split}
\end{equation}
where $$\Sigma=\begin{pmatrix}\Sigma_{\xb\xb} & \Sigma_{\xb\ub}\\\Sigma_{\ub\xb} & \Sigma_{\ub\ub}\end{pmatrix}.$$ 
Recall that in the online LQR problem, we deal with time-varying cost matrices $(\Qb_t,\Rb_t)$, and for any $t\in [T]$, the above SDP yields different solutions. In fact, for any feasible solution $\Sigma$ of the above SDP, a strongly stable controller $\Kb=\Sigma_{\xb\ub}^\top \Sigma_{\xb\xb}^{-1}$ can be extracted. The steady-state covariance matrix induced by this controller is also feasible for the SDP and its cost is at most that of $\Sigma$ (see Theorem 4.2 in \cite{cohen2018online}). Moreover, for any (slowly-varying) sequence of feasible solutions to the SDP, the induced controller sequence is sequentially strongly-stable. %This implies that the covariance matrix of the state rapidly converges to the steady-state. 

\subsection{Challenges of Distributed Online LQR for Unknown Dynamical Systems}
The works of \cite{cohen2018online} and \cite{chang2021distributed} tackle the centralized and decentralized online LQR, respectively. To keep the regret sublinear, the key idea in online LQR is to construct sequentially strongly stable controllers using online gradient descent (projected to the feasible set of SDP in \eqref{eq:SDP}). However, in our work, given that system parameters $(\Ab,\Bb)$ are unknown, the agents must perform a system identification first. The system identification step results in two challenges: (i) an exploration-exploitation trade-off to keep the regret sublinear, and (ii) different SDPs across agents as a result of decentralized estimation. The latter is particularly challenging, because as we can see in \eqref{eq:SDP}, each agent only has a local estimate of $(\Ab,\Bb)$, so the SDPs will have different feasible sets across agents, and we cannot directly apply distributed online optimization results on a common feasible set (e.g., \cite{yan2012distributed,shahrampour2018distributed}). In this work, we propose an algorithm (in the next section) for which we prove that an extracted controller based on a precise enough system estimates $(\widehat{\Ab},\widehat{\Bb})$ is strongly stable w.r.t. the system $(\Ab,\Bb)$. 

\section{Algorithm and Theoretical Results}
We now develop the distributed online LQR algorithm for unknown systems and study its theoretical regret bound. 
\subsection{Algorithm}
Our proposed method is outlined in Algorithm \ref{alg:Online Distributed LQR Control (unknown)}. In the first $T_0+1$ iterations, we need to collect data for the system identification. Suppose that each agent has access to a controller $\Kb_0$, which is $(\kappa_0,\gamma_0)$-strongly stable w.r.t. the system $(\Ab,\Bb)$. This controller can be different across agents, but for the presentation simplicity, we assume that agents use the same controller $\Kb_0$. The knowledge of such controller is a common assumption in centralized LQR (see e.g., \cite{dean2018regret,cohen2019learning}). In this period, agent $i$ at time $t$ applies the control $\ub_{i,t}\sim \Nc(\Kb_0\xb_{i,t}, 2\sigma^2\kappa^2_0\cdot\Ib)$, which prevents the state $\xb_{i,t}$ from going unbounded (lines 2-7). For the next $T_1$ iterations, all agents perform the system identification step by solving a distributed least-squares (LS) problem. In this step, the global LS problem is formed using the data collected by all agents, where each agent local cost is only based on its own collected data. Here, we can use any iterative distributed optimization algorithm to get precise enough system estimates. We employ the {\bf EXTRA} algorithm \cite{shi2015extra} since it achieves a geometric rate for strongly convex problems, and it can be implemented in a decentralized fashion (lines 8-16). After $T_0+T_1+1$ iterations, each agent $i$ at time $t$ runs a distributed online gradient descent on the SDP \eqref{eq:SDP}, where the local cost is defined w.r.t. matrices $\Qb_{i,t}$ and $\Rb_{i,t}$, and the feasible set is defined w.r.t. system estimates $(\widehat{\Ab}_{i,t},\widehat{\Bb}_{i,t})$. A control matrix $\Kb_{i,t}$ is then extracted from the update of $\Sigma_{i,t}$ and is used to determine the action. In particular, $\ub_{i,t}$ is sampled from a Gaussian distribution $\Nc(\Kb_{i,t}\xb_{i,t}, \Vb_{i,t})$, which entails $\mathrm{E}[\ub_{i,t}|\Fc_t]=\Kb_{i,t}\xb_{i,t}$, where $\Fc_t$ is the smallest $\sigma$-field containing the
information about all agents up to time $t$ (lines 17-26). The choice of $\Vb_{i,t}$ is due to a technical reason. It ensures the fast convergence of the covariance matrix of $\xb_{i,t}$ to the steady-state covariance matrix, when applying $\Kb_{i,t}$ to the underlying system $(\Ab,\Bb)$.

\subsection{Theoretical Result: Regret Bound}
In this section, we present our main theoretical result. By applying Algorithm \ref{alg:Online Distributed LQR Control (unknown)}, we show that for a multi-agent network of unknown LTI systems (with a connected communication graph), the individual regret of an arbitrary agent is upper-bounded by $O(T^{2/3}\log T)$, which implies that the agents {\it collectively} perform as well as the best fixed controller in hindsight for large enough $T$.
\begin{algorithm}[tb!]
   \caption{Online Distributed LQR Control with Unknown Dynamics}
   \label{alg:Online Distributed LQR Control (unknown)}
\begin{algorithmic}[1]
   \STATE {\bfseries Require:} number of agents $m$, doubly stochastic matrix $\Pb\in \mathrm{R}^{m\times m}$, parameter $\nu$, step size $\eta$, a $(\kappa_0,\gamma_0)$-strongly stable controller $\Kb_0$ w.r.t. system matrices $(\Ab,\Bb)$, covariance parameter of the noise $\sigma$, parameter $\vartheta$. 
   
   {\bfseries Initialize:} $\xb_{i,1} = 0,\forall i\in [m]$.
   
  \FOR{$t=1,2,\ldots,T_0 + 1$}
    \FOR{$i=1,2,\ldots,m$}
        \STATE Receive $\xb_{i,t}$
        \STATE Perform action  $\ub_{i,t}\sim \Nc(\Kb_0\xb_{i,t}, 2\sigma^2\kappa^2_0\cdot\Ib)$
    \ENDFOR
  \ENDFOR
  
  \STATE After the first $(T_0+1)$ iterations, each agent $i$ uses the collected data to form the local function
  \begin{equation*}
      f_i(\Ab,\Bb):=\sum_{t=1}^{T_0}\norm{[\Ab\:\Bb]\zb_{i,t} - \xb_{i,t+1}}^2 + \frac{\sigma^2\vartheta^{-2}}{m}\norm{[\Ab\:\Bb]}_F^2,
  \end{equation*}
  where $\zb_{i,t}=[\xb_{i,t}^{\top}\:\ub_{i,t}^{\top}]^{\top}$.

  \STATE  Choose the step size $\alpha$ following the result in \cite{shi2015extra} and denote %$(\Ib+\Pb)$ as $\hat{\Pb}$ and 
  $\Tilde{\Pb}:=\frac{\Ib+\Pb}{2}$. Denote by $\widehat{D}_{i}$ the agent $i$ vectorized system estimate $[\widehat{\Ab}_{i}\:\widehat{\Bb}_{i}]$. Apply {\bf EXTRA} to solve the global LS problem $\sum_{i=1}^m f_i(\Ab,\Bb)$ in a distributed fashion.
  \STATE  Randomly generate $\widehat{D}^0_i$ for all $i\in [m]$.
  \STATE $\forall i,\:\widehat{D}^1_i = \sum_{j=1}^m [\Pb]_{ji}\widehat{D}^0_j - \alpha\nabla f_i(\widehat{D}^0_i).$
  \FOR{$k=0,1,\ldots,T_1-1$}
    \FOR{$i=1,2,\ldots,m$}
        \STATE $\widehat{D}^{k+2}_i = \sum_{j=1}^m 2 [\Tilde{\Pb}]_{ji}\widehat{D}^{k+1}_j - \sum_{j=1}^m [\Tilde{\Pb}]_{ji}\widehat{D}^{k}_j - \alpha[\nabla f_i(\widehat{D}^{k+1}_i)-\nabla f_i(\widehat{D}^k_i)].$
    \ENDFOR
  \ENDFOR
  \STATE For all $i\in [m]$, transform the vectorized $\widehat{D}^{T_1+1}_i$ back to matrix form $[\widehat{\Ab}_{i,t}\:\widehat{\Bb}_{i,t}]$ for all $t\geq (T_0+1)+T_1$.
  
    \STATE Let $T_s:=(T_0+T_1+2)$.
    \STATE Initialize $\Sigma_{i,T_s}=\Sigma_{T_s}$ for any $i\in [m]$.
   \FOR{$t=T_s,\ldots,T$}
        % \STATE $\eta_t=\frac{1}{2\sqrt{t}}$
        \FOR{$i=1,2,\ldots,m$}
        \STATE Receive $\xb_{i,t}$
        \STATE Compute $\Kb_{i,t}=(\Sigma_{i,t})_{\ub \xb}(\Sigma_{i,t})_{\xb \xb}^{-1}$ and $ \Vb_{i,t}=(\Sigma_{i,t})_{\ub \ub}-\Kb_{i,t}(\Sigma_{i,t})_{\xb \xb}\Kb_{i,t}^\top$
        \STATE Perform $\ub_{i,t}\sim \Nc(\Kb_{i,t}\xb_{i,t}, \Vb_{i,t})$ and observe $\Qb_{i,t},\Rb_{i,t}$
        %\STATE Communicate $\Sigma_{i,t}$ with agents in the neighborhood and obtain their parameters
        %\STATE {\bf Optional:} Run {\bf EXTRA} to get new system estimates $(\widehat{\Ab}_{i,t+1}\:\widehat{\Bb}_{i,t+1})$.
        \STATE $\Sigma_{i,t+1} = \Pi_{\Sc^i_{t+1}}\left[\sum\limits_{j=1}^m[\Pb]_{ji}\Sigma_{j,t}-\eta \begin{pmatrix}\Qb_{i,t} & 0\\0 & \Rb_{i,t}
        \end{pmatrix}\right]$,
        where 
            \begin{equation*}
            \resizebox{.85\hsize}{!}{
            $\Sc^i_{t+1}:=\left\{\Sigma\in\mathrm{R}^{
            n\times n}\middle| 
            \begin{tabular}{@{}l@{}}$\Sigma \succeq 0,\quad \text{Tr}(\Sigma)\leq \nu,$\\ $\Sigma_{\xb\xb}=\widehat{\Cb}_{i,t+1}\Sigma \widehat{\Cb}_{i,t+1}^\top+\Wb$\end{tabular}
            \right\}$}, 
            \end{equation*}
             and $\widehat{\Cb}_{i,t+1}=[\widehat{\Ab}_{i,t+1}~\widehat{\Bb}_{i,t+1}]$.
        \ENDFOR
   \ENDFOR
\end{algorithmic}
\end{algorithm}

\begin{theorem}\label{T: Online Distributed LQR Controller (unknown)}
Assume that the network is connected, $\norm{[\Ab\:\Bb]}_F\leq \vartheta$,  $\text{Tr}(\Qb_{i,t}),\text{Tr}(\Rb_{i,t})\leq C$, $\text{Tr}(\Wb)\leq \lambda^2$ and $\Wb\succeq\sigma^2\Ib$. Given $\kappa\geq 1$ and $0\leq\gamma<1$, set $\nu=2\kappa^4\lambda^2/\gamma$ and step size $\eta=T^{-1/3}$.
If we run Algorithm \ref{alg:Online Distributed LQR Control (unknown)} with $T_0=T^{2/3}\log(T/\delta)$ and $T_1=O(\log(T^{1/3}))$, then with probability $(1-\delta)$, the individual regret of agent $j$ with respect to any $(\kappa,\gamma)$-strongly stable controller $\Kb^s$ is bounded as follows
$$\text{Regret}_T^j(\mathcal{A})=J^j_T(\mathcal{A})-J_T(\Kb^s)=O(T^{2/3}\log T),~$$~for $T$ large enough.
%\begin{equation*}
%    T\geq \text{poly}(\kappa, \gamma^{-1}, \kappa_0, \gamma_0^{-1}, n, m, \lambda, \sigma^{-1},C, (1-\beta), \log(\delta^{-1})).
%\end{equation*}
\end{theorem}

The exact expression of the lower bound for $T$ is given in \eqref{eq: Regret proof Unknown 11}, and it depends on $m,n,\kappa,\gamma,\beta,\delta$. The details of the proof are provided in the Appendix, and Section \ref{sketch} highlights the key technical challenges.

\begin{remark}
For online LQR control with known dynamics, \cite{cohen2018online} and \cite{chang2021distributed} prove regret bounds of $O(\sqrt{T})$ for centralized and distributed cases, respectively. However, in this work, since the system is unknown, agents need to compute system estimates first. This brings forward an exploration cost that increases the order of regret. In other words, agents objective is still to minimize the regret, but if they do not collect enough data, the estimation error propagates into the  exploitation phase, yielding a larger regret (in orders of magnitude).  
\end{remark}

\begin{remark}  
For online control with unknown dynamics, both \cite{hazan2020nonstochastic
} and \cite{simchowitz2020improper} consider the setup where costs are time-varying convex functions with adversarial noises, and they derive the regret bound of $O(T^{2/3})$ for fully observable systems and partially observable systems, respectively. In this work, we consider the distributed variant of online LQR control with stochastic noises and unknown dynamics. Our regret bound of $O(T^{2/3}\log T)$ is consistent with previous results on centralized problems in the convex setting (disregarding the log factor).
\end{remark}

\begin{figure*}[t]
    \centering
    \begin{subfigure}[t]{0.32\textwidth}
        \centering
        \includegraphics[width=\textwidth]{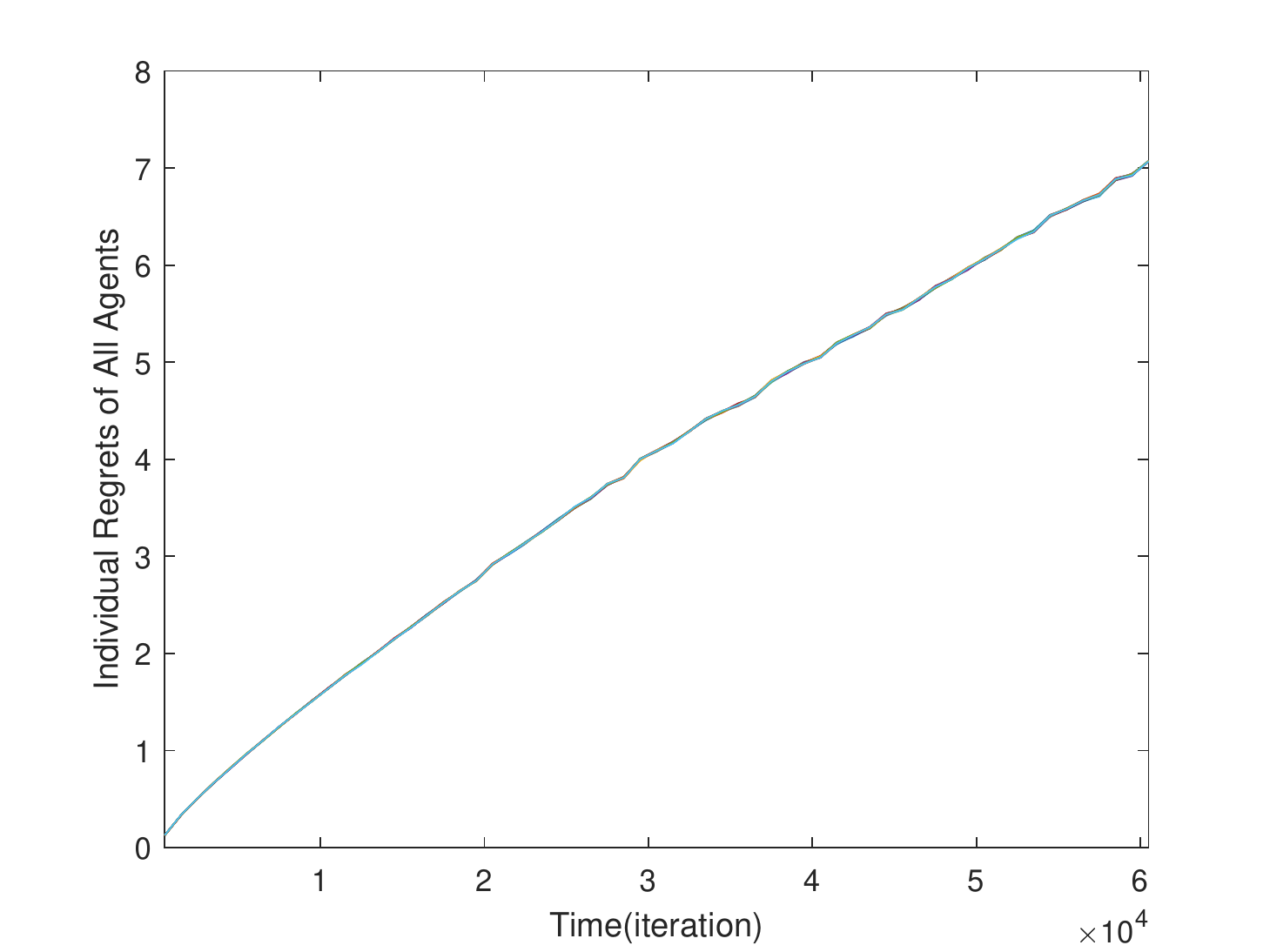}
        \caption{The plot of individual regrets of all agents over time.}
        \label{fig:individual regret}
    \end{subfigure}
    \hfill
    % \begin{subfigure}[t]{0.32\textwidth} 
    %     \centering
    %     \includegraphics[width=\textwidth]{regret_order.eps}
    %     \caption{When $T$ is large enough, the individual regrets divided by the regret order are bounded, which implies
    %     the individual regrets of all agents are of $O(T^{2/3}\log T)$.}
    %     \label{fig:individual regret order}
    % \end{subfigure}
    \begin{subfigure}[t]{0.32\textwidth}
        \centering
        \includegraphics[width=\textwidth]{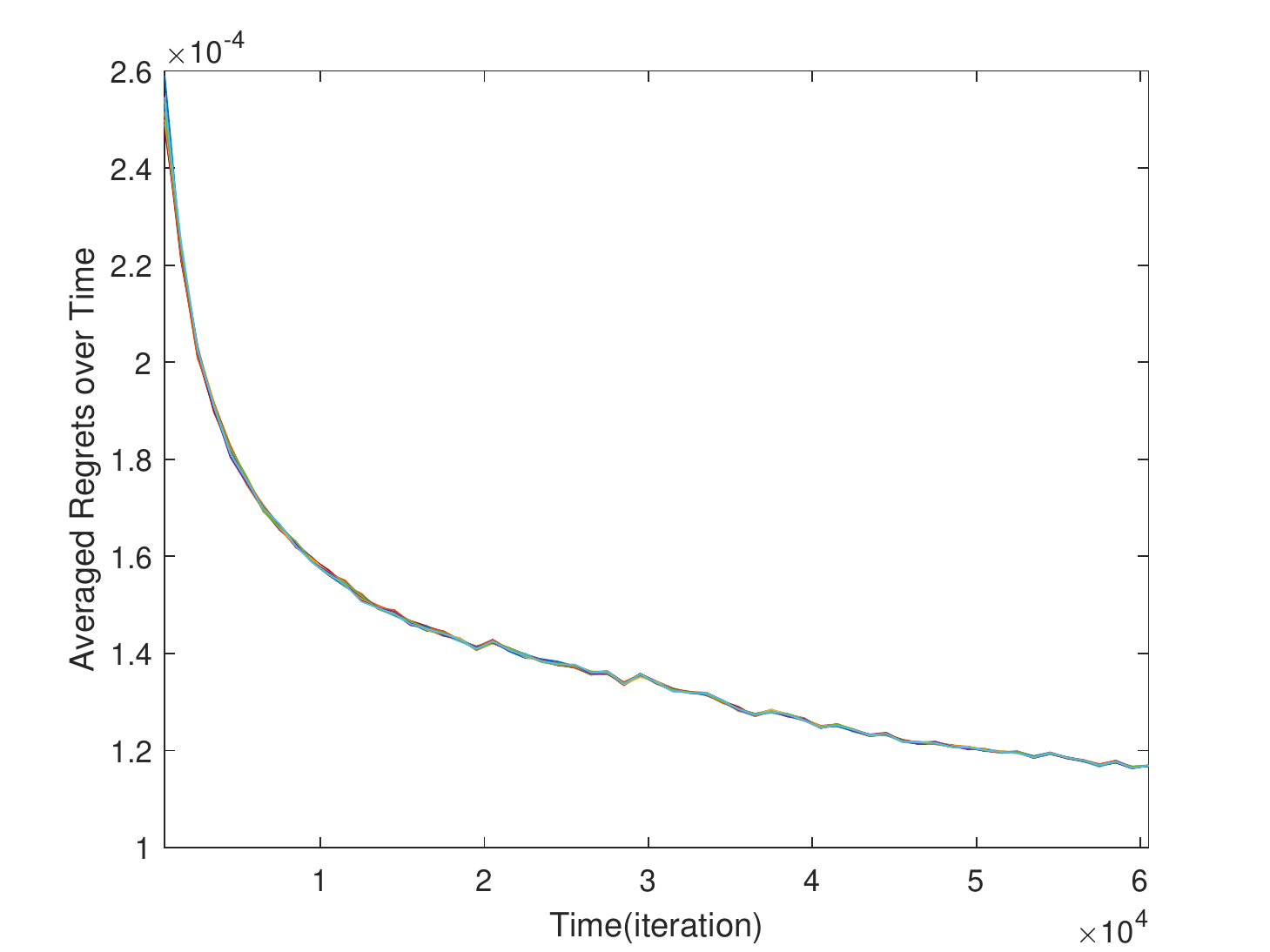}
        \caption{The temporal average of regret converges to zero for all agents.}
        \label{fig:consensus}
    \end{subfigure}
    \hfill
    \begin{subfigure}[t]{0.32\textwidth} 
        \centering
        \includegraphics[width=\textwidth]{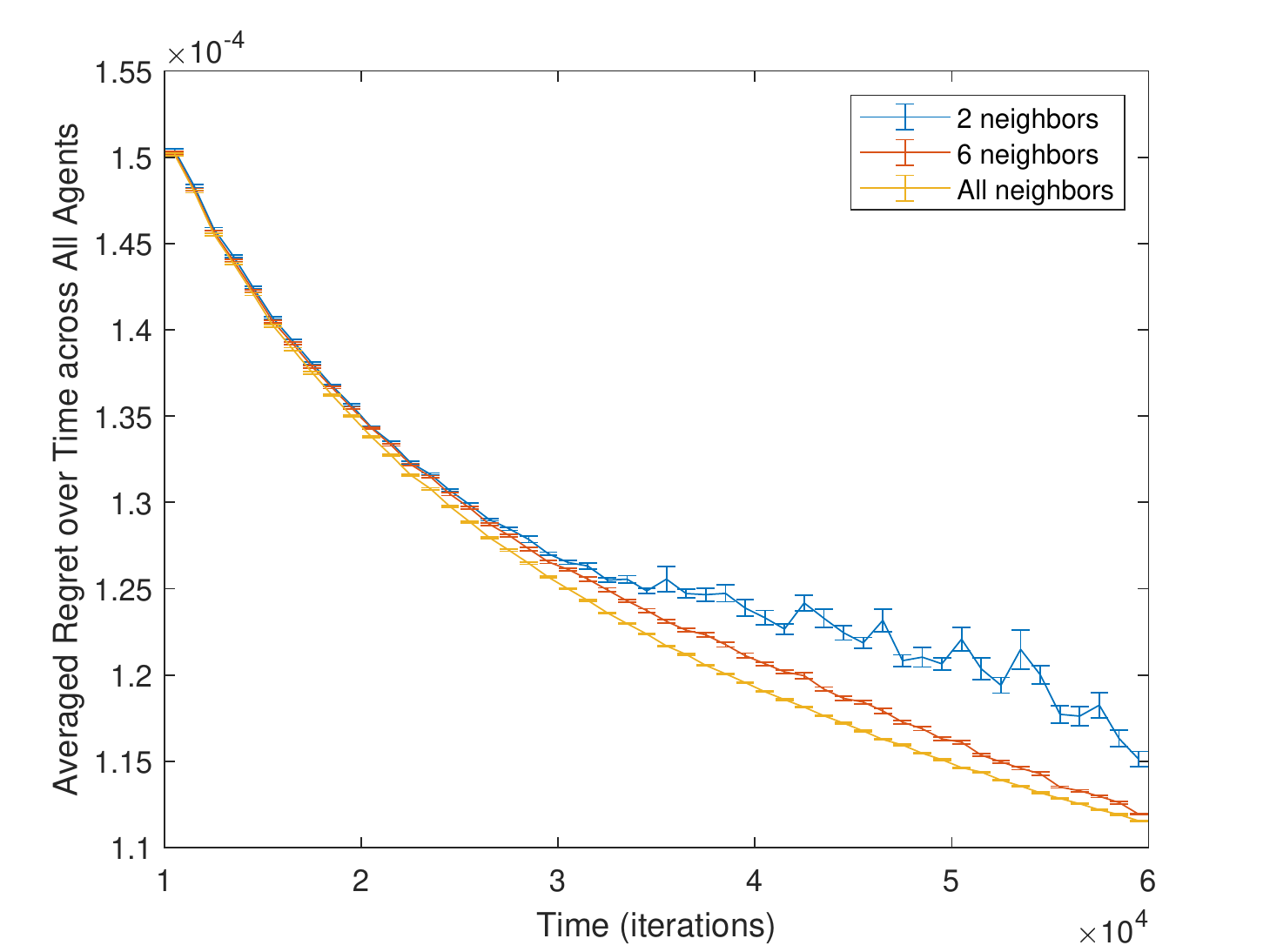}
        \caption{The averaged regrets over time for different networks: more connectivity results in smaller regret.}
        \label{fig:regret different nets}
    \end{subfigure}
    \caption{The individual regrets of all agents are shown to be sublinear.}
\end{figure*}

\begin{remark}\label{remark}
The exact expression of regret (given in \eqref{eq: Regret proof Unknown 10}) shows its dependence to $(1-\beta)^{-1}$, where $\beta$ is the second largest singular value of $\Pb$. This implies that when the network is well-connected (i.e., $\beta$ is smaller), the resulting bound is tighter. A smaller $\beta$ allows the Markov chain $\Pb$ to mix faster, which intuitively results in faster information propagation over the network of agents.
\end{remark}

\subsection{Key Technical Challenges in the Proof}\label{sketch}
The regret can be decomposed into three terms, where each term must be small enough to bound the regret. In \cite{cohen2018online}, two of these terms are bounded using the properties of strong stability and sequential strong stability, and one term is bounded using the standard regret bound for online gradient descent. In our setup, since $(\Ab,\Bb)$ is {\it unknown}, the agents cannot work with the ideal feasible set in \eqref{eq:SDP}, and as evident from line 25 of the algorithm, $\Sc^i_{t+1}$ is constructed only based on agent $i$ system estimate. This brings forward two challenges. (i) Agent $i$ constructs the controllers based on iterates $\Sigma_{i,t+1}$ that are not necessarily in the feasible set of \eqref{eq:SDP}, so we need to establish the stability properties of these controllers. (ii) The feasible sets are different across agents (i.e., $\Sc^i_{t+1}\neq \Sc^j_{t+1}$ for $i\neq j$), so we cannot directly apply distributed online optimization results on a common feasible set.

To tackle the first challenge, we first derive the bound on the precision of each agent system estimate based on the {\bf EXTRA} algorithm (see Lemma \ref{L: warm up estimation}) and combine that with statistical properties of centralized LS estimation using results of \cite{cohen2019learning}. We then establish that if each agent system estimate is close enough to the true system, a strongly stable policy w.r.t. the system estimate is also strongly stable w.r.t. the true system (see Lemma \ref{L:strong stability for a close system} and Lemma \ref{L: definiteness of two closed systems}). To address the second challenge, we use alternating projections to prove that a point in the feasible set of one agent is close enough to its projection to the feasible set of another agent, when estimates of these two agents are close. Then, in Theorem \ref{T: Dynamic Regret bound}, we show the contribution of distributed online optimization to the regret. We finally put together these results to prove our regret bound.

\section{Numerical Experiments}

% \begin{figure}[t] 
%     \includegraphics[width=.95\columnwidth]{regret.eps}
%     \caption{The plot of individual regrets of all agents over time.}
%     \label{fig:individual regret}
% \end{figure}

% \begin{figure}[t!] 
%     \includegraphics[width=.95\columnwidth]{regret_order.eps}
%     \caption{When $T$ is large enough, the individual regrets divided by the regret order are bounded, which implies
%     the individual regrets of all agents are of $\Tilde{O}(T^{2/3})$.}
%     \label{fig:individual regret order}
% \end{figure}

% \begin{figure}[t!] 
%     \includegraphics[width=.95\columnwidth]{averaged_regret.eps}
%     \caption{The averaged regrets over time for all agents converge as time grows.}
%     \label{fig:consensus}
% \end{figure}

We now provide numerical simulations verifying the theoretical guarantee of our algorithm.

\vspace{0.2cm}
\noindent
\textbf{Experiment Setup:} We first consider a network of $m=20$ agents, captured by a cyclic graph, where each agent has a self-weight of 0.6 and assigns the weight 0.2 to each of its two neighbors. The (hyper)-parameters are set as follows: $d=k=3$, $\kappa=1.5$, $\gamma=0.4$, $C=300$. We let matrices $\Ab=(1-2\gamma)\Ib$ and $\Bb=(\gamma/\kappa)\Ib$ to ensure the existence of a $(\kappa,\gamma)$-strongly stable controller.
For time-varying cost matrices, we set $\Qb_{i,t}$ (respectively, $\Rb_{i,t}$) as a diagonal matrix where each diagonal term is sampled from the uniform distribution over $[0,C/d]$ (respectively, $[0,C/k]$), so that $\text{Tr}(\Qb_{i,t}),\text{Tr}(\Rb_{i,t})\leq C$. The disturbance $\wb_{i,t}$ is sampled from a standard Gaussian distribution, and thus $\lambda^2=d=3$ and $\sigma^2=1$.

\vspace{0.2cm}
\noindent
\textbf{Simulation:} We simulate Algorithm \ref{alg:Online Distributed LQR Control (unknown)} for $T=1K,2K,3K,\ldots,60K$. For the benchmark, we set $\Kb^s$ as $(-\kappa)10^{-2}\Ib$ which is $(\kappa,\gamma)$-strongly stable w.r.t. $\Ab,\Bb$, and the resulting cumulative cost is small enough to be the benchmark. For the projection on the feasible set, we apply Dykstra's projection algorithm. Due to floating-point computations, $\Vb_{i,t}$ for action-sampling may not be positive semi-definite (PSD). Therefore, we address it by adding to $\Vb_{i,t}$ a small term ($(1e-15)\Ib$) to keep it PSD.  The entire process is repeated for $50$ Monte-Carlo simulations, and in the figures we present the averaged plots.% and the confidence interval over the trials.

% The total iteration number $T$ is set as 30 times of the theoretical lower bound in Theorem \ref{T: Online Distributed LQR Controller} in order to better see the performance. We let $\Kb^s=(1e-2)(-\kappa)\Ib$ which is $(\kappa,\gamma)$-strongly stable with $\Ab,\Bb$, and leads to a small enough cumulative cost to be the benchmark. Noting that we apply Dykstra's projection algorithm for the projection step, the matrix $\Vb_{i,t}$ for action-sampling may not be positive semi-definite (PSD) due to floating-point computations, so we do some tuning by adding to it a small term ($(1e-25)\Ib$) to keep it PSD. The parameters $\Sigma_{i,1}$ are identically initialized and the initial states of all agents are sampled from normal distribution. The entire process is repeated for 30 Monte-Carlo simulations. 

% \begin{figure}[t!] 
%     \includegraphics[width=.95\columnwidth]{averaged_regret_confidence_interval.eps}
%     \caption{The averaged regrets over time \& agents and its 95\% confidence interval.}
%     \label{fig:regret_confidence_interval}
% \end{figure}

\begin{table}[h!]
    \centering
    \begin{tabular}{|c|c|c|c|c|c|}
        \hline
         Iterations & 20K & 30K & 40K & 50K & 60K  \\
         \hline 
         Averaged Regret & 1.424 & 1.34 & 1.248 & 1.201 & 1.168 \\
         \hline
         Standard Error & 0.0087 & 0.0097 & 0.0052 & 0.0032 & 0.0037\\
         \hline
    \end{tabular}
    \caption{The mean and standard error of the averaged regret over time and agents ($\times 10^{-4}$).}
    \label{tab:regret_confidence_interval}
\end{table}

\noindent
\textbf{Performance:} I) Sublinearity of Regret: To verify the result of Theorem \ref{T: Online Distributed LQR Controller (unknown)}, %we plot the regret of agent 1 over time. Fig. \ref{fig:individual regret} shows that the regret is upper-bounded by a linear function. To better capture this, we also plot 
%the regret normalized by $\log(T)T^{2/3}$ in Fig. \ref{fig:individual regret order}. We observe that for large enough $T$, the normalized version of regret is bounded, which verifies that the regret is $O(\log(T)T^{2/3})$. To see the sub-linearity of individual regret,
in Fig. \ref{fig:consensus}, we present the averaged regret over time (i.e., individual regret divided by $T$), which is clearly decreasing over time. In Table \ref{tab:regret_confidence_interval}, we tabulate the averaged regrets (over time and agents) as well as their standard errors computed from $50$ trials for $T=20K,30K,40K,50K,60K$. We can see that $50$ trials is enough to obtain a small standard error. II)~Impact of Network Topology: To study the impact of network topology, we use three different networks: a cyclic graph with 2 neighbors (Net A), a cyclic graph with 6 neighbors (Net B) and a complete graph (Net C) where every entry of $\mathbf{P}$ is $\frac{1}{20}$. From Fig. \ref{fig:regret different nets}, we can see that the regret increases when $\beta$ is smaller (Net A $>$ Net B $>$ Net C). This result is consistent with Remark \ref{remark} on the impact of $\beta$.

% \begin{figure}[t!] 
%     \includegraphics[width=.95\columnwidth]{regret_diff_networks.eps}
%     \caption{The averaged regrets over time for different networks: more connectivity results in smaller regret.}
%     \label{fig:regret different nets}
% \end{figure}
\section{Conclusion}
In this paper, we considered the distributed online LQR problem with unknown LTI systems and time-varying quadratic cost functions. We developed a fully decentralized algorithm to estimate the unknown system and minimize the finite-horizon cost, which can be cast as a regret minimization. We proved that the individual regret, which is the performance of the control sequence of any agent compared to the best (linear and strongly stable) controller in hindsight, is upper bounded by $O(T^{2/3}\log T)$. Future directions include analyzing the {\it dynamic} regret defined w.r.t. the optimal (instantaneous) control policy in hindsight, investigating {\it coupled} time-varying cost functions, and analyzing the adversarial noise setup.

\section{Appendix}\label{appendix}
Let us start with an outline of the appendix as follows.
\begin{enumerate}
    \item In Section \ref{sec:A}, we discuss the connection of strong stability for two close enough LTI systems. We further discuss the relation between the steady-state covariance matrices of these systems.%two close enough LTI systems, when a strongly stable controller of one LTI system is applied to both dynamical systems.
    \item In Section \ref{sec:B},  based on the results in \cite{cohen2019learning} and the {\bf EXTRA} algorithm \cite{shi2015extra}, we quantify the precision of the system estimate of each agent. %as Algorithm \ref{alg:Online Distributed LQR Control (unknown)} is applied.
    \item As mentioned in the proof sketch, there is a term in the regret decomposition, related to distributed online optimization with different feasible sets across agents. In Section \ref{sec:C}, we provide auxiliary lemmas to bound this term in Theorem \ref{T: Dynamic Regret bound}.
    \item In Section \ref{sec:D}, we prove our main result (Theorem \ref{T: Online Distributed LQR Controller (unknown)}) by putting together the above results.
\end{enumerate}

\addtolength{\textheight}{-0cm}   

\begin{subsection} {Strong Stability of Two Similar LTI Systems}\label{sec:A}
Notice that we deal with a problem with unknown system dynamics, where each agent decides its control signal based on its own system estimates. Therefore, we should quantify how a strongly stable controller w.r.t. one system performs on another similar system, and how two steady-state covariance matrices of these two systems are related to each other.

\begin{lemma}\label{L:strong stability for a close system}
A linear policy $\Kb$ which is $(\kappa, \gamma)$-strongly stable ($\kappa\geq 1$) for the LTI system $(\Ab_1, \Bb_1)$, is also $(\kappa, \gamma-2\kappa^2\epsilon)$-strongly stable for the LTI system $(\Ab_2, \Bb_2)$ if $\norm{[\Ab_1\:\Bb_1] - [\Ab_2\:\Bb_2]}_{F}\leq \epsilon$ and $\epsilon < \frac{\gamma}{2\kappa^2}$.\end{lemma}
\begin{proof}
    Based on the definition of strong stability, we have $\Ab_1 + \Bb_1\Kb=\Hb\Lb\Hb^{-1}$ with $\norm{\Hb}\norm{\Hb^{-1}}\leq \kappa$ and $\norm{\Lb}\leq 1-\gamma$. Therefore, we have
    \begin{equation*}
    \begin{split}
        &\Ab_2+\Bb_2\Kb\\
        =&\Ab_1+\Bb_1\Kb + (\Ab_2 - \Ab_1) + (\Bb_2 - \Bb_1)\Kb\\
        =&\Hb\Lb\Hb^{-1} + \Hb\Hb^{-1}[(\Ab_2 - \Ab_1) + (\Bb_2 - \Bb_1)\Kb]\Hb\Hb^{-1}\\
        =&\Hb\big[\Lb + \Hb^{-1}[(\Ab_2 - \Ab_1) + (\Bb_2 - \Bb_1)\Kb]\Hb\big]\Hb^{-1}.
    \end{split}    
    \end{equation*}
    For the middle term%$\big[\Lb + \Hb^{-1}[(\Ab_2 - \Ab_1) + (\Bb_2 - \Bb_1)\Kb]\Hb\big]$
    , it can be shown that
    \begin{equation*}
    \begin{split}
        &\norm{\Lb + \Hb^{-1}[(\Ab_2 - \Ab_1) + (\Bb_2 - \Bb_1)\Kb]\Hb}\\
        \leq &\norm{\Lb} + \norm{\Hb^{-1}}\norm{(\Ab_2 - \Ab_1) + (\Bb_2 - \Bb_1)\Kb}\norm{\Hb}\\
        \leq & (1-\gamma) + \kappa(\norm{\Ab_2 - \Ab_1} + \norm{(\Bb_2 - \Bb_1)\Kb})\\
        \leq & (1-\gamma) + \kappa(\epsilon + \epsilon\kappa)\\
        \leq & (1-\gamma) + 2\kappa^2\epsilon.
    \end{split}    
    \end{equation*}
    Since $\gamma - 2\kappa^2\epsilon > 0$, based on the result above, we can see that the policy $\Kb$ is $(\kappa, \gamma - 2\kappa^2\epsilon)$-strongly stable w.r.t. the LTI system $(\Ab_2, \Bb_2)$ as $\Ab_2 + \Bb_2\Kb = \Hb\Lb_2\Hb^{-1}$ where $\Lb_2:=\Lb + \Hb^{-1}[(\Ab_2 - \Ab_1) + (\Bb_2 - \Bb_1)\Kb]\Hb$ and $\norm{\Lb_2}\leq 1-(\gamma - 2\kappa^2\epsilon)$. 
\end{proof}

\begin{lemma}\label{L: definiteness of two closed systems}
Given a linear policy $\Kb$ which is $(\kappa, \gamma)$-strongly stable ($\kappa\geq 1$) for the LTI system $(\Ab_1, \Bb_1)$, if $\norm{[\Ab_1\:\Bb_1] - [\Ab_2\:\Bb_2]}_{F}\leq \epsilon$ and $\epsilon < \frac{\gamma}{2\kappa^2}$, we have that
\begin{equation*}
    \Xb_1 \succeq \Xb_2 - 4\kappa^6\text{Tr}(\Wb)\frac{\epsilon[1-(\gamma-2\kappa^2\epsilon)]}{(\gamma-2\kappa^2\epsilon)(1-(1-\gamma)^2)}\cdot\Ib,
\end{equation*}
where $\Xb_1(\text{respectively}, \Xb_2)$ is the steady-state covariance matrix obtained by applying the controller $\ub_t=\Kb \xb_t$ on the system $\xb_{t+1} = \Ab_1\xb_{t} + \Bb_1\ub_{t} + \wb_{t}$ (\text{respectively}, $\xb_{t+1} = \Ab_2\xb_{t} + \Bb_2\ub_{t} + \wb_{t}$).
\end{lemma}
\begin{proof}
    Based on Lemma \ref{L:strong stability for a close system}, we know that the policy $\Kb$ is $(\kappa, \gamma)$ and $(\kappa, \gamma-2\kappa^2\epsilon)$-strongly stable for the systems $(\Ab_1,\Bb_1)$ and $(\Ab_2,\Bb_2)$, respectively. We then have
    \begin{equation*}
    \begin{split}
        \Xb_1 &= (\Ab_1 + \Bb_1\Kb)\Xb_1(\Ab_1 + \Bb_1\Kb)^{\top} + \Wb,\\
        \Xb_2 &= (\Ab_2 + \Bb_2\Kb)\Xb_2(\Ab_2 + \Bb_2\Kb)^{\top} + \Wb.
    \end{split}
    \end{equation*}
    Denoting $\Delta:=\Xb_1 - \Xb_2$, we get
    \begin{equation*}
    \begin{split}
        &\Xb_2 + \Delta\\ 
        = &(\Ab_1 + \Bb_1\Kb)(\Xb_2 + \Delta)(\Ab_1 + \Bb_1\Kb)^{\top} + \Wb\\
        = &(\Ab_2 + \Bb_2\Kb)\Xb_2(\Ab_2 + \Bb_2\Kb)^{\top}\\
        + &[(\Ab_1 + \Bb_1\Kb) - (\Ab_2 + \Bb_2\Kb)]\Xb_2(\Ab_2 + \Bb_2\Kb)^{\top}\\
        + &(\Ab_2 + \Bb_2\Kb)\Xb_2[(\Ab_1 + \Bb_1\Kb) - (\Ab_2 + \Bb_2\Kb)]^{\top}\\
        + &[\Ab_1-\Ab_2 + (\Bb_1-\Bb_2)\Kb ]\Xb_2[\Ab_1-\Ab_2 + (\Bb_1-\Bb_2)\Kb]^{\top}\\
        + &(\Ab_1 + \Bb_1\Kb)\Delta(\Ab_1 + \Bb_1\Kb)^{\top} + \Wb.
    \end{split}
    \end{equation*}
    From the equation above, we have
    \begin{equation*}
    \begin{split}
        &\Delta\\
        = &[(\Ab_1 + \Bb_1\Kb) - (\Ab_2 + \Bb_2\Kb)]\Xb_2(\Ab_2 + \Bb_2\Kb)^{\top}\\
        + &(\Ab_2 + \Bb_2\Kb)\Xb_2[(\Ab_1 + \Bb_1\Kb) - (\Ab_2 + \Bb_2\Kb)]^{\top}\\
        + &[\Ab_1-\Ab_2 + (\Bb_1-\Bb_2)\Kb ]\Xb_2[\Ab_1-\Ab_2 + (\Bb_1-\Bb_2)\Kb]^{\top}\\
        + &(\Ab_1 + \Bb_1\Kb)\Delta(\Ab_1 + \Bb_1\Kb)^{\top}\\
        \succeq &[(\Ab_1 + \Bb_1\Kb) - (\Ab_2 + \Bb_2\Kb)]\Xb_2(\Ab_2 + \Bb_2\Kb)^{\top}\\
        + &(\Ab_2 + \Bb_2\Kb)\Xb_2[(\Ab_1 + \Bb_1\Kb) - (\Ab_2 + \Bb_2\Kb)]^{\top}\\
        + &(\Ab_1 + \Bb_1\Kb)\Delta(\Ab_1 + \Bb_1\Kb)^{\top}.
    \end{split}
    \end{equation*}
    Denoting 
    \begin{equation*}
    \begin{split}
        \Psi&:=[(\Ab_1 + \Bb_1\Kb) - (\Ab_2 + \Bb_2\Kb)]\Xb_2(\Ab_2 + \Bb_2\Kb)^{\top}\\
        + &(\Ab_2 + \Bb_2\Kb)\Xb_2[(\Ab_1 + \Bb_1\Kb) - (\Ab_2 + \Bb_2\Kb)]^{\top}
    \end{split}    
    \end{equation*}
     and applying the inequality above recursively, we get
    \begin{equation*}
    \begin{split}
        \Delta &\succeq \Psi + (\Ab_1+\Bb_1\Kb)\Psi(\Ab_1+\Bb_1\Kb)^{\top} + \ldots \\
        &+ (\Ab_1+\Bb_1\Kb)^n\Psi((\Ab_1+\Bb_1\Kb)^{\top})^n\\
        &+ (\Ab_1+\Bb_1\Kb)^{n+1}\Delta((\Ab_1+\Bb_1\Kb)^{\top})^{n+1}.
    \end{split}
    \end{equation*}
    As $n\rightarrow \infty$, the spectral norm of the right hand side is upper bounded by %$\norm{\Psi}(1+\kappa^2(1-\gamma)^2 + \ldots + \kappa^2(1-\gamma)^{2n})\leq
    $\kappa^2\norm{\Psi}/[1-(1-\gamma)^2]$, which implies 
    \begin{equation}\label{eq:cov diff lower bound}
        \Delta \succeq \frac{-\kappa^2\norm{\Psi}}{1-(1-\gamma)^2}\Ib.
    \end{equation}
    We also have
    \begin{equation*}
    \begin{split}
        \norm{\Psi}\leq &2\norm{(\Ab_1 + \Bb_1\Kb) - (\Ab_2 + \Bb_2\Kb)}\norm{\Xb_2}\norm{\Ab_2 + \Bb_2\Kb}\\
        \leq &2(2\kappa\epsilon)[\kappa(1-(\gamma-2\kappa^2\epsilon))]\norm{\Xb_2}\\
        \leq &4\kappa^2\epsilon(1-(\gamma-2\kappa^2\epsilon))\frac{\kappa^2}{\gamma-2\kappa^2\epsilon}\text{Tr}(\Wb),
    \end{split}
    \end{equation*}
    where the second inequality is due to the strong stability of $\Kb$ w.r.t. $(\Ab_2,\Bb_2)$ and the third inequality is derived by applying Lemma 3.3 in \cite{cohen2018online}. Substituting the above upper bound into \eqref{eq:cov diff lower bound}, the result is proved.
\end{proof}
\begin{corollary}\label{C: definiteness of two closed systems}
Given a linear policy $\Kb$ which is $(\kappa, \gamma)$-strongly stable ($\kappa\geq 1$) for the LTI system $(\Ab_1, \Bb_1)$, if $\norm{[\Ab_1\:\Bb_1] - [\Ab_2\:\Bb_2]}_{F}\leq \epsilon$ and $\epsilon \leq \frac{\gamma}{4\kappa^2}$, we have the following result:
\begin{equation*}
    \Xb_1 \succeq \Xb_2 - \xi\epsilon\cdot\Ib,
\end{equation*}
where %$\xi:=8\kappa^6\text{Tr}(\Wb)\frac{(1-\gamma/2)}{\gamma(1-(1-\gamma)^2)}$.
$\xi:=\text{Tr}(\Wb)\frac{4\kappa^6}{\gamma^2}$.
\end{corollary}
\end{subsection}

\begin{subsection} {Precision of the System Estimates}\label{sec:B}
Since the final regret bound depends on the precision of all agents system estimates, in this section we show that with a specific choice of $T_0$ (iterations for collecting data), $T_1$ (iterations for performing {\bf EXTRA}) and $T$, the precision of each agent estimate and the distance between any two agents estimates are upper bounded by $O(T^{-1/3})$.

\begin{lemma}\label{L: warm up estimation}
    Assume $\norm{[\Ab\:\Bb]}_F\leq \vartheta $ and apply Algorithm \ref{alg:Online Distributed LQR Control (unknown)} with $T_0=T^{2/3}\log(T/\delta)$ and $T_1=O(\log(T^{1/3}))$ and $T_0\geq \max\{200(\log(12^n)+\log(\frac{3m}{\delta})),4\varrho+6m+3d\}$, where $\varrho:=\frac{m144\vartheta^2\kappa_0^4}{\gamma_0^2}(1+\vartheta^2\kappa_0^2)$. Then, for $t,t^{\prime}\in\{T_s,\ldots,T\}$ and $i,j\in [m]$, with probability $(1-\delta)$, we have 
    \begin{equation*}
        \norm{[\widehat{\Ab}_{i,t}\:\widehat{\Bb}_{i,t}]-[\Ab\:\Bb]}_F \leq (1+\frac{38\sqrt{2}n}{\sqrt{m}})T^{-1/3},
    \end{equation*}
    and
    \begin{equation*}
        \norm{[\widehat{\Ab}_{i,t}\:\widehat{\Bb}_{i,t}]-[\widehat{\Ab}_{j,t^{\prime}}\:\widehat{\Bb}_{j,t^{\prime}}]}_F \leq 2T^{-1/3}.
    \end{equation*}
\end{lemma}

\begin{proof}
    Let $V_0:=\sum_{i=1}^m\sum_{t=1}^{T_0}\zb_{i,t}\zb_{i,t}^{\top}$. Based on Theorem 20 in \cite{cohen2019learning}, we have the following relationships:
    \begin{itemize}
        \item With probability at least $(1-\delta/3)$,
        \begin{equation}\label{eq: warm up estimation 1}
            \text{Tr}(V_0)\leq (m T_0) \frac{144\sigma^2\kappa_0^4}{\gamma_0^2}(n+k\vartheta^2\kappa_0^2)\log(\frac{6m T_0}{\delta}).
        \end{equation}
        \item With probability at least $(1-\delta/3)$, based on our choice of $T_0$, we have that
        \begin{equation}\label{eq: warm up estimation 2}
            V_0\succeq\frac{mT_0\sigma^2}{80}\Ib.
        \end{equation}
        \item By solving the following LS problem
        \begin{equation*}
            \min_{[\Ab,\Bb]}\sum_{i=1}^m\sum_{t=1}^{T_0}\norm{[\Ab\:\Bb]\zb_{i,t} - \xb_{i,t+1}}^2 + \sigma^2\vartheta^{-2}\norm{[\Ab\:\Bb]}^2,
        \end{equation*}
        where the estimates $[\widehat{\Ab}\: \widehat{\Bb}]=(\sum_{i=1}^m\sum_{t=1}^{T_0}\xb_{i,t+1}\zb_{i,t}^{\top})V^{-1}$ and $V = V_0 + \sigma^2\vartheta^{-2}\Ib$, we have with probability at least $(1-\delta/3)$ (based on our choice of $T_0$)
        \begin{equation}\label{eq: warm up estimation 3}
            \text{Tr}(\widehat{\Delta}V\widehat{\Delta}^{\top})\leq 18\sigma^2n^2\log(T_0/\delta),
        \end{equation}
        where $\widehat{\Delta}=[\widehat{\Ab}\: \widehat{\Bb}]-[\Ab\:\Bb]$.
    \end{itemize}
    Combining \eqref{eq: warm up estimation 2} and \eqref{eq: warm up estimation 3}, we get
    \begin{equation*}
        \norm{\widehat{\Delta}}_F\leq \frac{38n}{\sqrt{m}}\sqrt{\frac{\log(T_0/\delta)}{T_0}}.
    \end{equation*}
    Since $T_0 = T^{2/3}\log(T/\delta)$, we have
    \begin{equation}\label{eq: warm up estimation 4}
        \norm{\widehat{\Delta}}_F\leq \frac{38\sqrt{2}n}{\sqrt{m}}T^{-1/3}. 
    \end{equation}
    After the first $(T_0+1)$ iterations, we apply the {\bf EXTRA} algorithm \cite{shi2015extra}%on the local system identification problem $f_i(\Ab,\Bb)$ of each agent $i$,
    , and based on Theorem 3.7 in \cite{shi2015extra}, we have the upper bound of the distance between $(\widehat{\Ab}_{i,t}, \widehat{\Bb}_{i,t})$ (agent $i$ estimation at iteration $t$) and $(\widehat{\Ab}, \widehat{\Bb})$, the solution of the global function $f(\Ab,\Bb)=\sum_{i=1}^m f_i(\Ab,\Bb)$ as follows. There exists $0<\tau < 1$ such that
    \begin{equation}\label{eq: warm up estimation 5}
        \norm{[\widehat{\Ab}_{i,t}\:\widehat{\Bb}_{i,t}]-[\widehat{\Ab}\:\widehat{\Bb}]}_F \leq \phi\tau^{t-(T_0+1)},\; \forall i,
    \end{equation}
    where $\phi$ is a constant. Based on \eqref{eq: warm up estimation 4}, \eqref{eq: warm up estimation 5} and our choice of $T_1$ ($T_1=(-\log\tau)^{-1}\log(\phi T^{1/3})$), for $t\in[T_s,\ldots,T]$ and $i,j\in[m]$, we have
    \begin{equation}\label{eq: warm up estimation 6}
    \begin{split}
        \norm{[\widehat{\Ab}_{i,t}\:\widehat{\Bb}_{i,t}]-[\Ab\:\Bb]}_F &\leq \norm{[\widehat{\Ab}_{i,t}\:\widehat{\Bb}_{i,t}]-[\widehat{\Ab}\:\widehat{\Bb}]}_F + \norm{\widehat{\Delta}}_F\\
        \leq & (1+\frac{38\sqrt{2}n}{\sqrt{m}})T^{-1/3},
    \end{split}
    \end{equation}
    and
    \begin{equation}\label{eq: warm up estimation 7}
    \begin{split}
        &\norm{[\widehat{\Ab}_{i,t}\:\widehat{\Bb}_{i,t}]-[\widehat{\Ab}_{j,t^{\prime}}\:\widehat{\Bb}_{j,t^{\prime}}]}_F\\
        \leq &\norm{[\widehat{\Ab}_{i,t}\:\widehat{\Bb}_{i,t}]-[\widehat{\Ab}\:\widehat{\Bb}]}_F + \norm{[\widehat{\Ab}_{j,t^{\prime}}\:\widehat{\Bb}_{j,t^{\prime}}]-[\widehat{\Ab}\:\widehat{\Bb}]}_F\\
        \leq & 2T^{-1/3}.
    \end{split}    
    \end{equation}
\end{proof}
\end{subsection}

\begin{subsection} {Bound of the Distributed Online SDP}\label{sec:C}
From the line 25 of Algorithm \ref{alg:Online Distributed LQR Control (unknown)}, we can see that the feasible set of SDP for each agent is based on the agent system estimate, so we cannot directly apply distributed online optimization results on a common feasible set. Here, we provide auxiliary results to bound the error due to distributed online optimization. In Lemma \ref{L: projection distance}, we use alternating projections to prove that a point in the feasible set of one agent is close enough to its projection to the feasible set of another agent, when estimates of these two agents are close. Then, in Theorem \ref{T: Dynamic Regret bound}, we show the contribution of distributed online optimization to the regret.%, 

\begin{lemma}\label{L: projection on an affine set as contraction}
    Consider an affine set $\Hc=\{\xb|\Eb\xb = \wb\}$, where $\Eb\in \Rb^{m\times n}$ is full row rank and $n> m$. For two points $\pb_1, \pb_2$, we have the following relationships.\\
    (i) $\norm{\Pi_{\Hc}(\pb_1) - \Pi_{\Hc}(\pb_2)}< \norm{\pb_1 - \pb_2}$ if $(\pb_1 - \pb_2)$ is not in the null space of $\Eb$.\\
    (ii) $\norm{\pb_1 - \Pi_{\Hc}(\pb_1)} = \norm{\pb_2 - \Pi_{\Hc}(\pb_2)}$ if $(\pb_1 - \pb_2)$ is in the null space of $\Eb$.
\end{lemma}
\begin{proof}
The orthogonal projection to affine set has the following closed-form \cite{meyer2000matrix}
\begin{equation}\label{eq: closed form affine set projection}
    \Pi_{\Hc}(\pb_1)=[\Ib-\Eb^\top(\Eb\Eb^\top)^{-1}\Eb]\pb_1+\Eb^\top(\Eb\Eb^\top)^{-1}\wb.
\end{equation}
The proof then follows immediately. 
\end{proof}

\begin{lemma}\label{L: projection distance}
    Suppose two system estimates $(\widehat{\Ab}_1, \widehat{\Bb}_1)$ and $(\widehat{\Ab}_2, \widehat{\Bb}_2)$ such that $\norm{[\widehat{\Ab}_1\;\widehat{\Bb}_1]-[\widehat{\Ab}_2\; \widehat{\Bb}_2]}_F\leq \epsilon_1$, and $\norm{[\widehat{\Ab}_1\;\widehat{\Bb}_1]-[\Ab\;\Bb]}_F, \norm{[\widehat{\Ab}_2\;\widehat{\Bb}_2]-[\Ab\;\Bb]}_F\leq \epsilon_2$. Let us denote $\Dc:=\{\Sigma|\Sigma\succeq0, \text{Tr}(\Sigma)\leq \nu\}$ and represent
    \begin{align*}
        \Hc_1&:=\{\Sigma_{\xb\xb}=\widehat{\Cb}_1\Sigma \widehat{\Cb}_1^\top+\Wb\}~~~~\widehat{\Cb}_1:=[\widehat{\Ab}_1 \; \widehat{\Bb}_1]\\
        \Hc_2&:=\{\Sigma_{\xb\xb}=\widehat{\Cb}_2\Sigma \widehat{\Cb}_2^\top+\Wb\}~~~~\widehat{\Cb}_2:=[\widehat{\Ab}_2 \; \widehat{\Bb}_2].
    \end{align*}
    Let us also define $\Sc_1:=\Dc\cap \Hc_1$ and $\Sc_2:=\Dc\cap \Hc_2$, respectively. Then, for any point $\Sigma_1\in\Sc_1$, we have that $\norm{\Sigma_1-\Pi_{\Sc_2}(\Sigma_1)}$ is $O(\epsilon_1)$.
\end{lemma}
\begin{proof}
    We can see that $\Hc_1$ can we written as
    \begin{equation*}
        \begin{bmatrix}\Ib & \mathbf{0}\end{bmatrix}\Sigma\begin{bmatrix}\Ib\\ \mathbf{0}\end{bmatrix} = \widehat{\Cb}_1\Sigma\widehat{\Cb}_1^\top + \Wb.
    \end{equation*}
    Denoting $\Db:=[\Ib \; \mathbf{0}]$, we have the vectorized version of the linear system above as 
    \begin{equation*}
        (\Db\otimes\Db-\widehat{\Cb}_1\otimes\widehat{\Cb}_1)\text{vec}(\Sigma) = \text{vec}(\Wb),
    \end{equation*}
    and we let $\Eb_1:=\Db\otimes\Db-\widehat{\Cb}_1\otimes\widehat{\Cb}_1$. Similarly, we can write $\Hc_2$ as $(\Eb_2)\text{vec}(\Sigma)=\text{vec}(\Wb)$. For the rest of the proof, we consider $\Sigma_1$ as $\text{vec}(\Sigma_1)$. Supposing both $\Eb_1$ and $\Eb_2$ are full row rank and applying \eqref{eq: closed form affine set projection}, for any point $\Sigma_1\in \Sc_1$ we have
    \begin{equation}\label{eq: projection distance 2}
    \begin{split} 
         &\Sigma_1 - \Pi_{\Hc_2}(\Sigma_1)=\Eb_2^{\top}(\Eb_2\Eb_2^{\top})^{-1}(\Eb_2-\Eb_1)\Sigma_1.
    \end{split}
    \end{equation}
    We know that $\Sigma_1\in\Sc_1$ and $[\widehat{\Ab}_2\;\widehat{\Bb}_2]$ has a finite norm, so there exists a constant upper-bounding $\norm{\Eb_2^{\top}(\Eb_2\Eb_2^{\top})^{-1}}\norm{\Sigma_1}$. To show $\norm{\Sigma_1 - \Pi_{\Hc_2}(\Sigma_1)}$ is $O(\epsilon_1)$, it is sufficient to show $\norm{\Eb_1 - \Eb_2}$ is $O(\epsilon_1)$. Based on the expressions of $\Eb_1$ and $\Eb_2$, we have that
    \begin{equation}
    \begin{split}
        &\norm{\Eb_2 - \Eb_1} \\
        = &\norm{\widehat{\Cb}_2\otimes\widehat{\Cb}_2 - \widehat{\Cb}_1\otimes\widehat{\Cb}_1}\\
        \leq &\norm{\widehat{\Cb}_2\otimes\widehat{\Cb}_2 - \widehat{\Cb}_2\otimes\widehat{\Cb}_1 + \widehat{\Cb}_2\otimes\widehat{\Cb}_1 - \widehat{\Cb}_1\otimes\widehat{\Cb}_1}_F\\
        %\leq &\sqrt{d(d+k)}(\norm{\widehat{\Cb}_1}_{l1}+\norm{\widehat{\Cb}_2}_{l1})\norm{\widehat{\Cb}_2 - \widehat{\Cb}_1}_F,
        \leq&(\norm{\widehat{\Cb}_1}_F+\norm{\widehat{\Cb}_2}_F)\norm{\widehat{\Cb}_2 - \widehat{\Cb}_1}_F,
    \end{split}
    \end{equation}
    which shows that $\norm{\Eb_2 - \Eb_1}$ is $O(\epsilon_1)$ due to the assumption that $\norm{\widehat{\Cb}_2 - \widehat{\Cb}_1}_F=\epsilon_1$. Therefore, we conclude that there exists a constant $\theta$ such that $\norm{\Sigma_1 - \Pi_{\Hc_2}(\Sigma_1)}\leq \theta\epsilon_1$.
    
      $\Sc_1$ and $\Sc_2$ are both non-empty and neither is a singleton. For any point $\Sigma_1\in\Sc_1$, if
     $\Pi_{\Hc_2}(\Sigma_1)\in \Dc$, by \eqref{eq: projection distance 2} and Pythagorean theorem, we have 
        \begin{equation}\label{eq: projection distance 9}
            \norm{\Sigma_1 - \Pi_{\Sc_2}(\Sigma_1)}\leq \norm{\Sigma_1 - \Pi_{\Hc_2}(\Sigma_1)} \leq \theta\epsilon_1,
        \end{equation}
        and the claim of lemma holds immediately. If $\Pi_{\Hc_2}(\Sigma_1)\notin \Dc$, we consider the process of applying alternating projections for $\Sigma_1$ on $\Sc_2$. Let $\yb_0=\Sigma_1$ and consider the following iterates
        \begin{equation}\label{eq: alt proj}
            \Pi_{\Dc}(\ab_{k-1}) = \yb_k ~~~~ \Pi_{\Hc_2}(\yb_k) = \ab_k,\: k=1,2,3,\ldots,
        \end{equation}
        where we denote the limit point by $\Sigma_1^{\prime}$. For the above sequences, based on the definition of projection we have
        \begin{equation*}
            \norm{\yb_{k+1} - \ab_{k}}^2 = \norm{\yb_{k+1} - \ab_{k+1}}^2 + \norm{\ab_{k+1} - \ab_{k}}^2,
        \end{equation*}
        which implies
        \begin{equation}\label{eq: projection distance 10}
        \begin{split}
            \norm{\yb_{k+1} - \ab_{k+1}}^2 &= \norm{\yb_{k+1} - \ab_{k}}^2 - \norm{\ab_{k+1} - \ab_{k}}^2\\
            & \leq  \norm{\yb_{k} - \ab_{k}}^2 - \norm{\ab_{k+1} - \ab_{k}}^2.
        \end{split}
        \end{equation}
        Without loss of generality, we assume $\norm{\ab_{k+1} - \ab_{k}} > 0$ for all $k$. If $\norm{\ab_{k+1} - \ab_{k}}=0$ at $k=\tilde{k}$, the sequence has converged in a finite number of steps (i.e., $\ab_k=\yb_k=\Sigma_1^{\prime}$ for $k\geq \tilde{k}$), and the following proof still holds. Assuming $\norm{\ab_{k+1} - \ab_{k}} > 0$, we can see from \eqref{eq: projection distance 10} that 
        \begin{equation}\label{eq: projection distance 11}
            \norm{\Sigma_1 - \ab_0}= \norm{\yb_0 - \ab_0} > \norm{\yb_k - \ab_k},\: k=1,2,\ldots
        \end{equation}
        If for any $k>0$, $(\Sigma_1 - \yb_k)$ is in the null space of $\Eb_2$, by Lemma \ref{L: projection on an affine set as contraction} we have 
        \begin{equation*}
            \norm{\Sigma_1 - \ab_0}=\norm{\Sigma_1 - \Pi_{\Hc_2}(\Sigma_1)} = \norm{\yb_k - \Pi_{\Hc_2}(\yb_k)},
        \end{equation*}
        which contradicts \eqref{eq: projection distance 11}. Therefore, we conclude that $(\Sigma_1 - \yb_k)$ is not in the null space of $\Eb_2$. Then, by Lemma \ref{L: projection on an affine set as contraction}, there exists a constant $0\leq \varphi_k<1$ such that %for the pair $(\Sigma_1, \yb_k)$
        \begin{equation}\label{eq: projection distance 12}
            \norm{\Pi_{\Hc_2}(\Sigma_1) - \Pi_{\Hc_2}(\yb_k)}\leq \varphi_k\norm{\Sigma_1 - \yb_k}.
        \end{equation}
        Now, define $k_0 := \lceil \log_{\rho}\frac{\theta\epsilon_1}{\norm{\Sigma_1 - \Sigma_1^{\prime}}}\rceil$, where $\rho$ denotes the linear convergence rate of alternating projections between two closed convex sets \cite{bauschke1993convergence}. In view of \eqref{eq: projection distance 12}, we have
        \begin{equation}\label{eq: projection distance 13}
            \norm{\Pi_{\Hc_2}(\Sigma_1) - \Pi_{\Hc_2}(\yb_k)}\leq \varphi\norm{\Sigma_1 - \yb_k},\: k\leq k_0,
        \end{equation}
        where $\varphi := \max\{\varphi_1,\varphi_2,\ldots,\varphi_{k_0}\}<1$. Also, the linear convergence rate along with our choice of $k_0$ guarantees
        \begin{align*}
           \norm{\ab_{k_0}-\Sigma_1^{\prime}}=\norm{\Pi_{\Hc_2}(\yb_{k_0})-\Sigma_1^{\prime}}\leq  \rho^{k_0}\norm{\Sigma_1 - \Sigma_1^{\prime}}\leq \theta\epsilon_1.
        \end{align*}
        Recalling \eqref{eq: alt proj} and combining \eqref{eq: projection distance 13} with above, we get
        \begin{equation*}
        \begin{split}
            &\norm{\Sigma_1 - \Pi_{\Sc_2}(\Sigma_1)}\leq \norm{\Sigma_1 - \Sigma_1^{\prime}}\\
            = &\norm{\Sigma_1 - \Pi_{\Hc_2}(\Sigma_1) + \Pi_{\Hc_2}(\Sigma_1) - \Pi_{\Hc_2}(\yb_{k_0}) + \Pi_{\Hc_2}(\yb_{k_0})  - \Sigma_1^{\prime}}\\
            \leq &\theta\epsilon_1 + \varphi\norm{\Sigma_1 - \yb_{k_0}} + \theta\epsilon_1\\
            = &2\theta\epsilon_1 + \varphi\norm{\Pi_{\Dc}(\Sigma_1) - \Pi_{\Dc}(\ab_{k_0-1})}\\
            \leq &2\theta\epsilon_1 + \varphi\norm{\Sigma_1 - \ab_{k_0-1}} \\
            =& 2\theta\epsilon_1 + \varphi\norm{\Sigma_1 - \Pi_{\Hc_2}(\Sigma_1) + \Pi_{\Hc_2}(\Sigma_1) - \Pi_{\Hc_2}(\yb_{k_0-1})}\\
            \leq &2\theta\epsilon_1 + \varphi\theta\epsilon_1 + \varphi^2\norm{\Sigma_1 - \yb_{k_0-1}}.
            \end{split}
        \end{equation*}
        Iteratively repeating above, we obtain 
        \begin{equation}\label{eq: projection distance 14}
        \begin{split}
            \norm{\Sigma_1 - \Pi_{\Sc_2}(\Sigma_1)} &\leq 2\theta\epsilon_1(1+\sum_{i=1}^{k_0}\varphi^i) + \varphi^{k_0+1}\norm{\Sigma_1 - \yb_{0}}\\
            &\leq \frac{2\theta\epsilon_1}{1-\varphi},
        \end{split}
        \end{equation}
        since $\norm{\Sigma_1 - \yb_{0}}=0$ based on the initialization of \eqref{eq: alt proj}. From \eqref{eq: projection distance 9} and \eqref{eq: projection distance 14}, we conclude that there exists a constant $\zeta$ such that $\norm{\Sigma_1 - \Pi_{\Sc_2}(\Sigma_1)} \leq \zeta\epsilon_1$.
\end{proof}
\begin{theorem}\label{T: Dynamic Regret bound}
Let Algorithm \ref{alg:Online Distributed LQR Control (unknown)} run with step size $\eta>0$ under the conditions of Theorem \ref{T: Online Distributed LQR Controller (unknown)}. Define 
$$
\Sc:=\Big\{\Sigma\succeq 0 \Big| \text{Tr}(\Sigma)\leq\nu \quad \Sigma_{\xb\xb}=[\Ab\:\Bb]\Sigma[\Ab\:\Bb]^\top+\Wb\Big\}.
$$
Then, for any $\Sigma\in\Sc$, the following quantity
\begin{equation*}
    \sum_{t=T_s}^T\begin{pmatrix}\Qb_{t} & 0\\0 & \Rb_{t}
        \end{pmatrix}\bullet\Sigma_{j,t} - \sum_{t=T_s}^T\begin{pmatrix}\Qb_{t} & 0\\0 & \Rb_{t}
        \end{pmatrix}\bullet\Sigma,
\end{equation*}
is $O(T^{2/3}+T\eta+\frac{T^{1/3}}{\eta})$ for any $j\in [m]$.
\end{theorem}
\begin{proof} 
For the presentation simplicity let
 $$f^i_t(\Sigma):=\begin{pmatrix}\Qb_{i,t} & 0\\0 & \Rb_{i,t}
        \end{pmatrix}\bullet\Sigma~~~~~f_t(\Sigma):=\sum_{i=1}^m f^i_t(\Sigma),$$ 
        and define $g^i_t:=\nabla f^i_t$. Observe that $\norm{g^i_t}_F\leq 2C$ since $\text{Tr}(\Qb_{i,t})\leq C$ and  $\text{Tr}(\Rb_{i,t})\leq C$.  
        For the rest of the proof, with a slight abuse of notation, we use the vectorized versions of matrices  $\Sigma_{i,t}$, $\Sigma$, and $g^i_t$ using 
        the same notation. We then have
        \begin{equation}\label{eq: Dynamic Regret Bound (unknown) 1}
        \begin{split}
            \Tilde{\Sigma}_{i,t+1} &= \sum_{j=1}^m[\Pb]_{ji}\Sigma_{j,t} - \eta g^i_t
            \\
            \Sigma_{i,t+1} &= \Pi_{\Sc^i_{t+1}}(\Tilde{\Sigma}_{i,t+1}).
        \end{split}    
        \end{equation}
Define $r^i_t:=\Sigma_{i,t} - \Tilde{\Sigma}_{i,t}$. For $t\in[T_s,\ldots,T]$ we can bound $r^i_{t+1}$ as follows:
\begin{equation}\label{eq: Dynamic Regret Bound (unknown) 2}
\begin{split}
    \vphantom{\sum_{j=1}^m}\norm{r^i_{t+1}}=&\norm{\Tilde{\Sigma}_{i,t+1} -\Pi_{\Sc^i_{t+1}}(\Tilde{\Sigma}_{i,t+1})}\\
    \vphantom{\sum_{j=1}^m}\leq &\Big\|\Tilde{\Sigma}_{i,t+1} - \sum_{j=1}^m[\Pb]_{ji}\Pi_{\Sc^i_{t+1}}(\Sigma_{j,t})\Big\|\\
    \vphantom{\sum_{j=1}^m}
    =&\Big\|\sum_{j=1}^m[\Pb]_{ji}[\Sigma_{j,t} - \Pi_{\Sc^i_{t+1}}(\Sigma_{j,t})]- \eta g^i_t\Big\|\\
    \leq &\zeta 2T^{-1/3} + \eta\norm{g^i_t}\leq 2\zeta T^{-1/3} + 2\eta C,
\end{split}
\end{equation}
where the first inequality is due to the properties of projection to a convex set, and the second inequality can be derived by applying Lemma \ref{L: warm up estimation} and Lemma \ref{L: projection distance} with $\epsilon_1=2T^{-1/3}$. For the sake of simplicity, we define the following matrices
\begin{equation}\label{eq:notation}
\begin{split}
    \mathbf{\Sigma}_t &:= [\Sigma_{1,t},\ldots,\Sigma_{m,t}],\;\Tilde{\mathbf{\Sigma}}_t := [\Tilde{\Sigma}_{1,t},\ldots,\Tilde{\Sigma}_{m,t}]\\
    G_t &:= [g^1_t,\ldots,g^m_t],\;~~~~~R_t := [r^1_t,\ldots,r^m_t].
\end{split}
\end{equation}
Then, for the iterate $\Sigma_t:=\frac{1}{m}\sum_{i=1}^m\Sigma_{i,t}$, we have the following relationship
\begin{equation}\label{eq: Dynamic Regret Bound (unknown) 3}
\begin{split}
    \Sigma_{t+1} &= \frac{1}{m}\mathbf{\Sigma}_{t+1}\1=\frac{1}{m}(\mathbf{\Sigma}_t\Pb - \eta G_t+R_{t+1})\1\\
    &=\frac{1}{m}\mathbf{\Sigma}_t\1 - \frac{\eta}{m}G_t\1 + \frac{1}{m}R_{t+1}\1\\
    &=\Sigma_t - \frac{\eta}{m}\sum_{i=1}^m g^i_t + \frac{1}{m}\sum_{i=1}^m r^i_{t+1}.
\end{split}
\end{equation}
For any $\Sigma\in\Sc$ we have that 
\begin{equation}\label{eq: Dynamic Regret Bound (unknown) 4}
\begin{split}
    &\norm{\Sigma_{t+1}-\Sigma}^2=\norm{\Sigma_{t}-\Sigma}^2+\frac{1}{m^2}\Big\|\sum_{i=1}^m(r^i_{t+1}-\eta g^i_t)\Big\|^2\\
    &~~~~~~-\frac{2\eta}{m}\sum_{i=1}^m\langle\Sigma_t-\Sigma,g^i_t\rangle+\frac{2}{m}\sum_{i=1}^m\langle\Sigma_t-\Sigma,r^i_{t+1}\rangle.
\end{split}    
\end{equation}
We now derive an upper bound of \eqref{eq: Dynamic Regret Bound (unknown) 4} for $t\in[T_s,\ldots,T]$. Based on \eqref{eq: Dynamic Regret Bound (unknown) 2} and the fact that $\norm{ g^i_t}\leq 2C$, we have that 
%To bound the term $\frac{1}{m^2}\norm{\sum_{i=1}^m(r^i_{t+1}-\eta g^i_t)}^2$, based on \eqref{eq: Dynamic Regret Bound (unknown) 2} we have
% \begin{equation}\label{eq: Dynamic Regret Bound (unknown) 5}
% \begin{split}
%     \frac{1}{m^2}\norm{\sum_{i=1}^m(r^i_{t+1}+\eta g^i_t)}^2\leq &\frac{1}{m^2}\left(\sum_{i=1}^m(\norm{r^i_{t+1}}+\eta\norm{ g^i_t})\right)^2\\
%     \leq &\frac{1}{m^2}\left(\sum_{i=1}^m(\zeta T^{-1/3}+2\eta\norm{g^i_t})\right)^2
% \end{split}
% \end{equation}
\begin{equation}\label{eq: Dynamic Regret Bound (unknown) 5}
\begin{split}
   \frac{1}{m^2}\Big\|\sum_{i=1}^m(r^i_{t+1}-\eta g^i_t)\Big\|^2\leq &\frac{1}{m^2}\left[\sum_{i=1}^m(\norm{r^i_{t+1}}+\eta\norm{ g^i_t})\right]^2\\
    \leq &(2\zeta T^{-1/3} + 4\eta C)^2.
\end{split}
\end{equation}
Based on the definition of $f^i_t(\Sigma)$, we have that
%\begin{equation}\label{eq: Dynamic Regret Bound (unknown) 6}
% \begin{split}
%     -\langle\Sigma_t-\Sigma,g^i_t\rangle &=-\langle\Sigma_t-\Sigma_{i,t},g^i_t\rangle + \langle\Sigma-\Sigma_{i,t},g^i_t\rangle\\
%     \leq &\norm{\Sigma_t-\Sigma_{i,t}}\norm{g^i_t} + f^i_t(\Sigma) - f^i_t(\Sigma_{i,t})\\
%     = &\norm{\Sigma_t-\Sigma_{i,t}}\norm{g^i_t} + f^i_t(\Sigma) - f^i_t(\Sigma_t)\\
%     + &f^i_t(\Sigma_t) - f^i_t(\Sigma_{i,t})\\
%     \leq &\norm{\Sigma_t-\Sigma_{i,t}}\norm{g^i_t} + f^i_t(\Sigma) - f^i_t(\Sigma_t)\\
%     + &\langle\Sigma_t-\Sigma_{i,t},\bar{g}^i_t\rangle\\
%     \leq &\norm{\Sigma_t-\Sigma_{i,t}}(\norm{g^i_t}+\norm{\bar{g}^i_t}) + f^i_t(\Sigma) - f^i_t(\Sigma_t).
% \end{split}    
% \end{equation}
\begin{equation}\label{eq: Dynamic Regret Bound (unknown) 6}
    \sum_{i=1}^m\langle\Sigma-\Sigma_t,g^i_t\rangle =\sum_{i=1}^m f^i_t(\Sigma) - f^i_t(\Sigma_t)=f_t(\Sigma) - f_t(\Sigma_t).
\end{equation}
%Summing \eqref{eq: Dynamic Regret Bound (unknown) 6} over $i$, we bound $-\sum_{i=1}^m\langle\Sigma_t-\Sigma,g^i_t\rangle$ by
%\begin{equation}\label{eq: Dynamic Regret Bound (unknown) 7}
%\begin{split}
%\sum_{i=1}^m \norm{\Sigma_t-\Sigma_{i,t}}(\norm{g^i_t}+\norm{\bar{g}^i_t}) + f_t(\Sigma) - f_t(\Sigma_t).
%\end{split}   
%\end{equation}
For the term $\langle\Sigma_t-\Sigma,r^i_{t+1}\rangle$, we derive the upper bound as follows
\begin{equation}\label{eq: Dynamic Regret Bound (unknown) 8}
\begin{split}
    &\langle\Sigma_t-\Sigma,r^i_{t+1}\rangle=\langle\Sigma_t-\Tilde{\Sigma}_{i,t+1},r^i_{t+1}\rangle + \langle\Tilde{\Sigma}_{i,t+1}-\Sigma,r^i_{t+1}\rangle\\
     &~~~~~~~\leq\norm{\Sigma_t-\Tilde{\Sigma}_{i,t+1}}\norm{r^i_{t+1}}+ \langle\Pi_{\Sc^i_{t+1}}(\Sigma)-\Sigma, r^i_{t+1}\rangle\\
    &~~~~~~~+\langle\Tilde{\Sigma}_{i,t+1}-\Pi_{\Sc^i_{t+1}}(\Sigma), r^i_{t+1}\rangle\\
    &~~~~~~~\leq\norm{\Sigma_t-\Tilde{\Sigma}_{i,t+1}}\norm{r^i_{t+1}} + \langle\Pi_{\Sc^i_{t+1}}(\Sigma)-\Sigma, r^i_{t+1}\rangle\\
    &~~~~~~~\leq \norm{r^i_{t+1}}\big(\norm{\Sigma_t-\Tilde{\Sigma}_{i,t+1}}+\norm{\Pi_{\Sc^i_{t+1}}(\Sigma)-\Sigma}\big)\\
    &~~~~~~~\leq 2\big
    (\zeta T^{-1/3} + \eta C\big)\big(\norm{\Sigma_t-\Tilde{\Sigma}_{i,t+1}}+\zeta'T^{-1/3}\big),
\end{split}    
\end{equation}
where the second inequality is due to the fact that $\langle\Tilde{\Sigma}^i_{t+1}-\Pi_{\Sc^i_{t+1}}(\Sigma), r^i_{t+1}\rangle$ is non-positive based on the properties of a projection operator, and the last inequality is based on \eqref{eq: Dynamic Regret Bound (unknown) 2} as well as Lemma \ref{L: projection distance} with $\epsilon_1=(1+\frac{38\sqrt{2}n}{\sqrt{m}})T^{-1/3}$ and $\zeta':=\zeta(1+\frac{ 38\sqrt{2}n}{\sqrt{m}})$.

Substituting \eqref{eq: Dynamic Regret Bound (unknown) 5}, \eqref{eq: Dynamic Regret Bound (unknown) 6} and \eqref{eq: Dynamic Regret Bound (unknown) 8} into \eqref{eq: Dynamic Regret Bound (unknown) 4} and rearranging it, we can get

\begin{equation}\label{eq: Dynamic Regret Bound (unknown) 9}
\begin{split}
    &f_t(\Sigma_t) - f_t(\Sigma)\leq\frac{m}{\eta}(\zeta T^{-1/3} + 2\eta C)^2\\
    &~~~~~~+\frac{m}{2\eta}\big(\norm{\Sigma_{t}-\Sigma}^2-\norm{\Sigma_{t+1}-\Sigma}^2\big)\\
    %+&\sum_{i=1}^m \norm{\Sigma_t-\Sigma_{i,t}}(\norm{g^i_t}+\norm{\bar{g}^i_t})\\
    &~~~~~~+\frac{2\zeta T^{-1/3} + 2\eta C}{\eta}\sum_{i=1}^m\big(\norm{\Sigma_t-\Tilde{\Sigma}_{i,t+1}}+\zeta ^{\prime}T^{-1/3}\big).
\end{split}    
\end{equation}
Adding and subtracting $f_t(\Sigma_{j,t})$ on the left hand side and observing that $f_t$ is Lipschitz continuous with the constant $2mC$, we get for any $j\in [m]$
\begin{equation}\label{eq: Dynamic Regret Bound (unknown) 10}
\begin{split}
    &f_t(\Sigma_{j,t}) - f_t(\Sigma)\leq\frac{m}{\eta}(\zeta T^{-1/3} + 2\eta C)^2\\
    &~~~~~~+\frac{m}{2\eta}\big(\norm{\Sigma_{t}-\Sigma}^2-\norm{\Sigma_{t+1}-\Sigma}^2\big)\\
    %+&\sum_{i=1}^m \norm{\Sigma_t-\Sigma_{i,t}}(\norm{g^i_t}+\norm{\bar{g}^i_t})\\
    &~~~~~~+\frac{2\zeta T^{-1/3} + 2\eta C}{\eta}\sum_{i=1}^m\big(\norm{\Sigma_t-\Tilde{\Sigma}_{i,t+1}}+\zeta ^{\prime}T^{-1/3}\big)\\
    &~~~~~~+2mC\norm{\Sigma_{j,t}-\Sigma_t}
\end{split}
\end{equation}
Applying Lemma \ref{L: consecutive distance} on \eqref{eq: Dynamic Regret Bound (unknown) 10} and summing it over $t\in[T_s,\ldots,T]$, we have
\begin{equation}\label{eq: Dynamic Regret Bound (unknown) 11}
\begin{split}
    &\sum_{t=T_s}^Tf_t(\Sigma_{j,t}) - f_t(\Sigma)\leq\frac{mT}{\eta}(\zeta T^{-1/3} + 2\eta C)^2\\
    %+&\sum_{i=1}^m \norm{\Sigma_t-\Sigma_{i,t}}(\norm{g^i_t}+\norm{\bar{g}^i_t})\\
    &+\frac{2\zeta T^{-1/3} + 2\eta C}{\eta}mT\big((2\zeta T^{-1/3} + 4\eta C)\frac{\sqrt{m}}{1-\beta}+\zeta^{\prime}T^{-1/3}\big)\\
    &+2mCT(2\zeta T^{-1/3} + 4\eta C)\frac{\sqrt{m}}{1-\beta}+\frac{m}{2\eta}\norm{\Sigma_{T_s}-\Sigma}^2,
\end{split}
\end{equation}
which is $O(T^{2/3}+T\eta+\frac{T^{1/3}}{\eta})$. If $\eta$ is $O(T^{-1/3})$, the above bound is $O(T^{2/3})$.
\end{proof}

\begin{lemma}\label{L: consecutive distance}
Let Algorithm \ref{alg:Online Distributed LQR Control (unknown)} run with step size $\eta>0$ and define $\Sigma_t:=\frac{1}{m}\sum_{i=1}^m\Sigma_{i,t}$. Under the conditions of Theorem \ref{T: Online Distributed LQR Controller (unknown)}, we have that
\begin{equation*}
\begin{split}
    \norm{\Sigma_t - \Sigma_{i,t}}\leq (2\zeta T^{-1/3} + 4\eta C)\frac{\sqrt{m}}{1-\beta},
\end{split}
\end{equation*}
for $t\in[T_s+1,\ldots,T]$, and
\begin{equation*}
\begin{split}
    &\norm{\Sigma_t - \Tilde{\Sigma}_{i,t+1}}
    \leq (2\zeta T^{-1/3} + 4\eta C)\frac{\sqrt{m}}{1-\beta},
\end{split}
\end{equation*}
for $t\in[T_s,\ldots,T]$.
\end{lemma}
\begin{proof}
Using the notations defined in \eqref{eq:notation}, we have the following relation $\mathbf{\Sigma}_{t}=\mathbf{\Sigma}_{t-1}\Pb - \eta G_{t-1}+R_{t}$. Unwinding the equation, we get
\begin{equation}\label{eq: consecutive distance 1}
\mathbf{\Sigma}_t = \mathbf{\Sigma}_{T_s}\Pb^{t-T_s} - \eta\sum_{k=1}^{t-T_s}G_{t-k}\Pb^{k-1} + \sum_{k=1}^{t-T_s}R_{t-k+1}\Pb^{k-1}.
\end{equation}
Since $\Pb$ is doubly stochastic, we have $\Pb^k\1=\1$ for all $k\geq 1$. Using \eqref{eq: Dynamic Regret Bound (unknown) 2}, the geometric mixing bound of $\Pb$, and the gradient norm bound of $2C$, we get
\begin{equation*}
\begin{split}
    &\norm{\Sigma_t - \Sigma_{i,t}}=\norm{\mathbf{\Sigma}_t(\frac{1}{m}\1-\eb_i)}\\
    \leq &\norm{\Sigma_{T_s}-\mathbf{\Sigma}_{T_s}[\Pb^{t-T_s}]_{:,i}}+\eta\sum_{k=1}^{t-T_s}\norm{G_{t-k}(\frac{1}{m}\1-[\Pb^{k-1}]_{:,i})}\\
    +&\sum_{k=1}^{t-T_s}\norm{R_{t-k+1}(\frac{1}{m}\1-[\Pb^{k-1}]_{:,i})}\\
    \leq &\eta\sum_{k=1}^{t-T_s}2C\sqrt{m}\beta^{k-1}+ \sum_{k=1}^{t-T_s}(2\zeta T^{-1/3} + 2\eta C)\sqrt{m}\beta^{k-1}\\
    \leq &(2\zeta T^{-1/3} + 4\eta C)\frac{\sqrt{m}}{1-\beta},
\end{split}    
\end{equation*}
where $\norm{\Sigma_{T_s}-\mathbf{\Sigma}_{T_s}[\Pb^{t-T_s}]_{:,i}}=0$ by the initialization. By the same token, 
\begin{equation*}
\begin{split}
    &\norm{\Sigma_t - \Tilde{\Sigma}_{i,t+1}}=\norm{\frac{1}{m}\mathbf{\Sigma}_t \1-(\mathbf{\Sigma}_t\Pb-\eta G_t)\eb_i}\\
    =&\norm{\mathbf{\Sigma}_t(\frac{1}{m}\1-\Pb\eb_i)+\eta G_t\eb_i}\\
    \leq &\norm{\Sigma_{T_s} - \mathbf{\Sigma}_{T_s}[\Pb^{t-T_s+1}]_{:,i}} + \eta\sum_{k=1}^{t-T_s}\norm{G_{t-k}(\frac{1}{m}\1-[\Pb^{k}]_{:,i})}\\
    +&\sum_{k=1}^{t-T_s}\norm{R_{t-k+1}(\frac{1}{m}\1-[\Pb^{k}]_{:,i})} + \norm{\eta g^i_t}\\
    \leq &2\eta C+\eta\sum_{k=1}^{t-T_s}2C\sqrt{m}\beta^{k} + \sum_{k=1}^{t-T_s}(2\zeta T^{-1/3} + 2\eta C)\sqrt{m}\beta^{k}\\
    \leq &\sum_{k=0}^{t-T_s} (2\zeta T^{-1/3} + 4\eta C)\sqrt{m}\beta^k\\
    \leq &(2\zeta T^{-1/3} + 4\eta C)\frac{\sqrt{m}}{1-\beta}.
\end{split}    
\end{equation*}
\end{proof}
\end{subsection}
\begin{subsection} {Bound of the Individual Regret}\label{sec:D}
\noindent
{{\bf\itshape Proof of Theorem \ref{T: Online Distributed LQR Controller (unknown)}: }}
Based on Algorithm \ref{alg:Online Distributed LQR Control (unknown)}, the first $(T_0+T_1+1)$ iterations are used to collect data and obtain the system estimates. The regret of this part is at most $O(T_0+T_1+1)$, where $T_0$ and $T_1$ are specified in Lemma \ref{L: warm up estimation}.
Let us now denote 
$$
\Lb_{i,t}:=\begin{pmatrix}\Qb_{i,t}&0\\0&\Rb_{i,t}\end{pmatrix}~~~~\text{and}~~~~\Lb_{t}:=\begin{pmatrix}\Qb_{t}&0\\0&\Rb_{t}\end{pmatrix},
$$
where $\Lb_{t}=\sum_{i=1}^m\Lb_{i,t}$. Also let,
    \begin{align*}
        \widehat{\Sigma}_{j,t}&:=\mathrm{E}\left[[\xb_{j,t}^\top\:\ub_{j,t}^\top]^\top[\xb_{j,t}^\top\:\ub_{j,t}^\top]\right]\\
    \widehat{\Sigma}_t^s&:=\mathrm{E}\left[[\xb_t^{s\top}\:\ub_t^{s\top}]^\top[\xb_t^{s\top}\:\ub_t^{s\top}]\right],
    \end{align*}
where $\ub_t^s = \Kb^s\xb_t^s$ (the control sequence generated by the benchmark controller $\Kb^s$ and the corresponding state sequence). Recalling $T_s=T_0+T_1+2$, we write the regret as
\begin{equation}\label{eq: Regret proof Unknown 1}
\begin{split}
    &\sum_{t=T_s}^T\Lb_t\bullet\widehat{\Sigma}_{j,t} - \sum_{t=T_s}^T\Lb_t\bullet\widehat{\Sigma}_t^s
    =\sum_{t=T_s}^T\Lb_t\bullet(\widehat{\Sigma}_{j,t}-\Sigma_{j,t})\\ + &~~~~~~~\sum_{t=T_s}^T\Lb_t\bullet(\Sigma_{j,t} -\Sigma^s)+\sum_{t=T_s}^T\Lb_t\bullet(\Sigma^s-\widehat{\Sigma}_t^s),
\end{split}    
\end{equation}
where $\Sigma^s$ is the steady-state covariance matrix induced by $\Kb^s$, and $\Sigma_{j,t}$ is generated by Algorithm \ref{alg:Online Distributed LQR Control (unknown)}. Now, we show how each term in \eqref{eq: Regret proof Unknown 1} is bounded.\\\\
{\bf (I)} For the term $\sum_{t=T_s}^T\Lb_t\bullet\Sigma_{j,t} - \sum_{t=T_s}^T\Lb_t\bullet\Sigma^s$:\\
We know $\Kb^s$ is $(\kappa,\gamma)$-strongly stable w.r.t. $(\Ab,\Bb)$, and based on Lemma 3.3 in \cite{cohen2018online}, it can be shown that $\text{Tr}(\Sigma^s)=\text{Tr}(\Sigma^s_{\xb\xb})+\text{Tr}(\Sigma^s_{\ub\ub})\leq 2\kappa^4\lambda^2/\gamma=\nu$, which ensures that $\Sigma^s$ is feasible to $\Sc$. From Theorem \ref{T: Dynamic Regret bound}, we have

\begin{equation}\label{eq: Regret proof Unknown 3}
\begin{split}
    &\sum_{t=T_s}^T\Lb_t\bullet\Sigma_{j,t} - \sum_{t=T_s}^T\Lb_t\bullet\Sigma^s\leq \frac{mT}{\eta}(\zeta T^{-1/3} + 2\eta C)^2\\
    +&\frac{2\zeta T^{-1/3} + 2\eta C}{\eta}mT\big((2\zeta T^{-1/3} + 4\eta C)\frac{\sqrt{m}}{1-\beta}+\zeta ^{\prime}T^{-1/3}\big)\\
    +&2mCT(2\zeta T^{-1/3} + 4\eta C)\frac{\sqrt{m}}{1-\beta}+\frac{m}{2\eta}\norm{\Sigma_{T_s}-\Sigma}^2.
\end{split}    
\end{equation}

\noindent
{\bf (II)} For the term $\sum_{t=T_s}^T\Lb_t\bullet(\widehat{\Sigma}_{j,t}-\Sigma_{j,t})$:\\
Based on Lemma \ref{L: warm up estimation}, we have for $t\in[T_s,\ldots,T]$, $$\norm{[\widehat{\Ab}_{j,t}\:\widehat{\Bb}_{j,t}]-[\Ab\:\Bb]}_F\leq \epsilon:= (1+\frac{38\sqrt{2}n}{\sqrt{m}})T^{-1/3},$$ 
with probability at least $1-\delta$. Let $\bar{\kappa}:=\frac{\sqrt{\nu}}{\sigma}$ and $\bar{\gamma}:=\frac{1}{2\bar{\kappa}^2}$. Based on Lemma 4.3 in \cite{cohen2018online}, it can be shown that $\Kb_{j,t}$ is $(\bar{\kappa},\bar{\gamma})$-strongly stable w.r.t. $(\widehat{\Ab}_{j,t},\widehat{\Bb}_{j,t})$, and it is  $(\bar{\kappa}, \frac{\bar{\gamma}}{2})$-strongly stable w.r.t. $(\Ab,\Bb)$ based on Lemma \ref{L:strong stability for a close system} and our choice of $T$ such that $\epsilon\leq \frac{\bar{\gamma}}{4\bar{\kappa}^2}$.
Based on Corollary \ref{C: definiteness of two closed systems}, we get 
\begin{equation*}
    \Xb_{j,t}^1\succeq\Xb_{j,t}^2-\xi\epsilon\cdot\Ib,
\end{equation*}
where $\Xb^1_{j,t}$ and $\Xb^2_{j,t}$ are the steady-state covariance matrices of applying $\Kb_{j,t}$ on the linear systems
$(\widehat{\Ab}_{j,t}, \widehat{\Bb}_{j,t})$ and $(\Ab,\Bb)$, respectively. $\xi$ is $\text{Tr}(\Wb)\frac{4\bar{\kappa}^6}{\bar{\gamma}^2}$. From Algorithm \ref{alg:Online Distributed LQR Control (unknown)}, we have
\begin{equation*}
\begin{split}
    \Sigma_{j,t} = & \begin{pmatrix}(\Sigma_{j,t})_{\xb\xb}&(\Sigma_{j,t})_{\xb\ub}\\(\Sigma_{j,t})_{\ub\xb}&(\Sigma_{j,t})_{\ub\ub}\end{pmatrix}\\
    =&\begin{pmatrix}(\Sigma_{j,t})_{\xb\xb}&(\Sigma_{j,t})_{\xb\xb}\Kb_{j,t}^\top\\\Kb_{j,t}(\Sigma_{j,t})_{\xb\xb}&\Kb_{j,t}(\Sigma_{j,t})_{\xb\xb}\Kb_{j,t}^\top\end{pmatrix}+\begin{pmatrix}0&0\\0&\Vb_{j,t}\end{pmatrix} 
\end{split}
\end{equation*}
and
\begin{equation*}
\begin{split}
    \widehat{\Sigma}_{j,t}=\begin{pmatrix}(\widehat{\Sigma}_{j,t})_{\xb\xb}&(\widehat{\Sigma}_{j,t})_{\xb\xb}\Kb_{j,t}^\top\\\Kb_{j,t}(\widehat{\Sigma}_{j,t})_{\xb\xb}&\Kb_{j,t}(\widehat{\Sigma}_{j,t})_{\xb\xb}\Kb_{j,t}^\top\end{pmatrix} + \begin{pmatrix}0&0\\0&\Vb_{j,t}\end{pmatrix}.
\end{split}    
\end{equation*}
Based on above we have 
\begin{equation}\label{eq: Regret proof Unknown 4}
\begin{split}
    &\Lb_{i,t}\bullet(\widehat{\Sigma}_{j,t}-\Sigma_{j,t})\\
    =&(\Qb_{i,t} + \Kb_{j,t}^\top\Rb_{i,t}\Kb_{j,t})\bullet\big((\widehat{\Sigma}_{j,t})_{\xb\xb}-(\Sigma_{j,t})_{\xb\xb}\big)\\
    \leq &(\Qb_{i,t} + \Kb_{j,t}^\top\Rb_{i,t}\Kb_{j,t})\bullet\big((\widehat{\Sigma}_{j,t})_{\xb\xb}-\Xb^1_{j,t}\big)\\
    \leq &(\Qb_{i,t} + \Kb_{j,t}^\top\Rb_{i,t}\Kb_{j,t})\bullet\big((\widehat{\Sigma}_{j,t})_{\xb\xb}-\Xb^2_{j,t} + \xi\epsilon\cdot\Ib\big)\\
    \leq &\text{Tr}(\Qb_{i,t} + \Kb_{j,t}^\top\Rb_{i,t}\Kb_{j,t})\big(\norm{(\widehat{\Sigma}_{j,t})_{\xb\xb}-\Xb^2_{j,t}} + \xi\epsilon\big)\\
    \leq &C(1+\bar{\kappa}^2)\big(\norm{(\widehat{\Sigma}_{j,t})_{\xb\xb}-\Xb^2_{j,t}} + \xi\epsilon\big),
\end{split}    
\end{equation}
where the first inequality can be derived based on the proof of Theorem 4.2 in \cite{cohen2018online}, and the last inequality comes from the fact that $\text{Tr}(\Qb_{i,t}),\text{Tr}(\Rb_{i,t})\leq C$ and $\norm{\Kb_{j,t}}\leq \bar{\kappa}$.\\\\
Based on Lemma \ref{L: consecutive distance} and \eqref{eq: Dynamic Regret Bound (unknown) 2}, we can derive 
\begin{equation*}
    \norm{\Sigma_{j,t+1}-\Sigma_{j,t}}\leq \frac{3\sqrt{m}}{1-\beta}(2\zeta T^{-1/3} + 4\eta C).
\end{equation*}
Choose $\eta$ and $T$ to ensure $\norm{\Sigma_{j,t+1}-\Sigma_{j,t}}\leq \frac{\bar{\gamma}\sigma^2}{2}$; it can then be shown that $\{\Kb_{j,t}\}_{t\geq(T_s)}$ are $(\bar{\kappa}, \bar{\gamma}/2)$-sequentially strongly stable w.r.t. $(\Ab,\Bb)$ based on the similar derivation of Lemma 4.4 in \cite{cohen2018online}.
Then, we have 
\begin{equation}\label{eq: Regret proof Unknown 5}
\begin{split}
    \norm{(\widehat{\Sigma}_{j,t})_{\xb\xb}-\Xb^2_{j,t}}\leq &\bar{\kappa}^2e^{-\frac{\bar{\gamma}}{2}(t-T_{s})}\norm{(\widehat{\Sigma}_{j,T_{s}})_{\xb\xb}-\Xb^2_{j,T_{s}}}\\
    + &\frac{4\bar{\kappa}^2}{\bar{\gamma}}\frac{3\sqrt{m}}{1-\beta}(2\zeta T^{-1/3} + 4\eta C).
\end{split}
\end{equation}
Substituting \eqref{eq: Regret proof Unknown 5} into \eqref{eq: Regret proof Unknown 4} and summing over $t\in[T_{s},\ldots,T]$, we can get
\begin{equation}\label{eq: Regret proof Unknown 6}
\begin{split}
    &\sum_{t=T_{s}}^T\Lb_{i,t}\bullet(\widehat{\Sigma}_{j,t}-\Sigma_{j,t})\\
    \leq &C(1+\bar{\kappa}^2)\bar{\kappa}^2\norm{(\widehat{\Sigma}_{j,T_{s}})_{\xb\xb}-\Xb^2_{j,T_{s}}}\sum_{t=T_{s}}^Te^{-\frac{\bar{\gamma}}{2}(t-T_{s})}\\
    + &\sum_{t=T_{s}}^TC(1+\bar{\kappa}^2)(\xi\epsilon)\\
    + &\sum_{t=T_s}^T C(1+\bar{\kappa}^2)\frac{4\bar{\kappa}^2}{\bar{\gamma}}\frac{3\sqrt{m}}{1-\beta}(2\zeta T^{-1/3} + 4\eta C)\\
    \leq &C(1+\bar{\kappa}^2)(\bar{\kappa}^2+\frac{2\bar{\kappa}^2}{\bar{\gamma}})\norm{(\widehat{\Sigma}_{j,T_{s}})_{\xb\xb}-\Xb^2_{j,T_{s}}}\\
    + &C(1+\bar{\kappa}^2)(\xi\epsilon)T\\
    + &C(1+\bar{\kappa}^2)\frac{4\bar{\kappa}^2}{\bar{\gamma}}\frac{3\sqrt{m}}{1-\beta}(2\zeta T^{-1/3} + 4\eta C)T,
\end{split}
\end{equation}
where the second inequality comes from the fact that $\sum_{t=1}^Te^{-\alpha t}\leq \int_0^\infty e^{-\alpha t}dt=1/\alpha$ for $\alpha>0$. Summing \eqref{eq: Regret proof Unknown 6} over $i$, the result is obtained.\\\\
{\bf (III)} For the term $\sum_{t=T_s}^T\Lb_t\bullet(\Sigma^s-\widehat{\Sigma}_t^s)$:\\
By denoting $\Sigma^s = \begin{pmatrix}\Sigma^s_{\xb\xb} & \Sigma^s_{\xb\xb}\Kb^{s\top}\\\Kb^s\Sigma^s_{\xb\xb}&\Kb^s\Sigma^s_{\xb\xb}\Kb^{s\top}\end{pmatrix}$ and $\widehat{\Sigma}^s_t = \begin{pmatrix}(\widehat{\Sigma}^s_t)_{\xb\xb} & (\widehat{\Sigma}^s_t)_{\xb\xb}\Kb^{s\top}\\\Kb^s(\widehat{\Sigma}^s_t)_{\xb\xb}&\Kb^s(\widehat{\Sigma}^s_t)_{\xb\xb}\Kb^{s\top}\end{pmatrix}$, we have
\begin{equation}\label{eq: Regret proof Unknown 7}
\begin{split}
    &\Lb_t\bullet(\Sigma^s-\widehat{\Sigma}^s_t)\\
    =&\sum_{i=1}^m(\Qb_{i,t}+\Kb^s\Rb_{i,t}\Kb^{s\top})\bullet\left(\Sigma^s_{\xb\xb}-(\widehat{\Sigma}^s_t)_{\xb\xb}\right)\\
    \leq &\sum_{i=1}^m \text{Tr}(\Qb_{i,t}+\Kb^s\Rb_{i,t}\Kb^{s\top})\norm{\Sigma^s_{\xb\xb}-(\widehat{\Sigma}^s_t)_{\xb\xb}}\\
    \leq &mC(1+\kappa^2)\norm{\Sigma^s_{\xb\xb}-(\widehat{\Sigma}^s_t)_{\xb\xb}},
\end{split}
\end{equation}
where the second inequality comes from the fact that $\text{Tr}(\Qb_{i,t}),\text{Tr}(\Rb_{i,t})\leq C$ and $\Kb^s$ is $(\kappa,\gamma)$-strongly stable. Based on Lemma 3.2 in \cite{cohen2018online}, we get
\begin{equation}\label{eq: Regret proof Unknown 8}
    \norm{(\widehat{\Sigma}^s_t)_{\xb\xb}-\Sigma^s_{\xb\xb}}\leq \kappa^2 e^{-2\gamma(t-T_s)}\norm{(\widehat{\Sigma}^s_{T_s})_{\xb\xb}-\Sigma^s_{\xb\xb}}.
\end{equation}
Substituting \eqref{eq: Regret proof Unknown 8} into \eqref{eq: Regret proof Unknown 7} and summing over $t\in[T_s,\ldots,T]$, we have 
\begin{equation}\label{eq: Regret proof Unknown 9}
\begin{split}
    &\sum_{t=T_s}^T\Lb_t\bullet(\Sigma^s - \widehat{\Sigma}_t^s)\\
    \leq &mC(1+\kappa^2)\kappa^2\norm{(\widehat{\Sigma}^s_{T_s})_{\xb\xb}-\Sigma^s_{\xb\xb}}\sum_{t=T_s}^Te^{-2\gamma(t-T_s)}\\
    \leq &\frac{mC(\kappa^2 + \kappa^4)(1+2\gamma)}{2\gamma}\norm{(\widehat{\Sigma}^s_{T_s})_{\xb\xb}-\Sigma^s_{\xb\xb}}.
\end{split}
\end{equation}
Substituting \eqref{eq: Regret proof Unknown 3}, \eqref{eq: Regret proof Unknown 6} and \eqref{eq: Regret proof Unknown 9} into \eqref{eq: Regret proof Unknown 1}, we get

\begin{equation}\label{eq: Regret proof Unknown 10}
\begin{split}
    &\sum_{t=T_s}^T\Lb_t\bullet\widehat{\Sigma}_{j,t} - \sum_{t=T_s}^T\Lb_t\bullet\widehat{\Sigma}_t^s\leq \frac{mT}{\eta}(\zeta T^{-1/3} + 2\eta C)^2\\
    +&\frac{2\zeta T^{-1/3} + 2\eta C}{\eta}mT\big((2\zeta T^{-1/3} + 4\eta C)\frac{\sqrt{m}}{1-\beta}+\zeta^{\prime}T^{-1/3}\big)\\
    +&2mCT(2\zeta T^{-1/3} + 4\eta C)\frac{\sqrt{m}}{1-\beta}+\frac{m}{2\eta}\norm{\Sigma_{T_s}-\Sigma}^2\\
    + &mC(1+\bar{\kappa}^2)(\bar{\kappa}^2+\frac{2\bar{\kappa}^2}{\bar{\gamma}})\norm{(\widehat{\Sigma}_{j,T_{s}})_{\xb\xb}-\Xb^2_{j,T_{s}}}\\
    + &mC(1+\bar{\kappa}^2)\xi(1+\frac{38\sqrt{2}n}{\sqrt{m}})T^{2/3}\\
    + &mC(1+\bar{\kappa}^2)\frac{4\bar{\kappa}^2}{\bar{\gamma}}\frac{3\sqrt{m}}{1-\beta}(2\zeta T^{-1/3} + 4\eta C)T\\
    + &\frac{mC(\kappa^2 + \kappa^4)(1+2\gamma)}{2\gamma}\norm{(\widehat{\Sigma}^s_{T_s})_{\xb\xb}-\Sigma^s_{\xb\xb}}.
\end{split}
\end{equation}

By setting $\eta=T^{-1/3}$, we can observe \eqref{eq: Regret proof Unknown 10} is $O(T^{2/3})$. Together with the linear regret in the first $(T_0+T_1+1)$ iterations, which is $O(T^{2/3}\log T)$, we conclude that the total regret is $O(T^{2/3}\log T)$. Note that $T$ is chosen such that the conditions of Lemma \ref{L: warm up estimation} are satisfied;  $(1+\frac{38\sqrt{2}n}{\sqrt{m}})T^{-1/3}\leq \frac{\bar{\gamma}}{4\bar{\kappa}^2}$; $\frac{3\sqrt{m}}{1-\beta}(2\zeta  + 4C)T^{-1/3}\leq \frac{\bar{\gamma}\sigma^2}{2}$.

\begin{equation}\label{eq: Regret proof Unknown 11}
\resizebox{0.95\hsize}{!}{$
\begin{aligned}
    T\geq&\max\bigg\{\left[(1+\frac{38\sqrt{2}n}{\sqrt{m}})\frac{32\kappa^8\lambda^4}{\gamma^2\sigma^4}\right]^3, \left[\frac{3\sqrt{m}}{1-\beta}(2\zeta + 4C)\frac{8\kappa^4\lambda^2}{\sigma^4\gamma}\right]^3,\\
    &\big[200(\log(12^n)+\log(\frac{3m}{\delta}))\big]^{3/2},\big[4\varrho+6m+3d\big]^{3/2}\bigg\}.
\end{aligned}
$}
\end{equation}
~\qed
\end{subsection}

%%%%%%%%%%%%%%%%%%%%%%%%%%%%%%%%%%%%%%%%%%%%%%%%%%%%%%%%%%%%%%%%%%%%%%%%%%%%%%%%

\bibliographystyle{IEEEtran}
\bibliography{references}

%%%%%%%%%%%%%%%%%%%%%%%%%%%%%%%%%%%%%%%%%%%%%%%%%%%%%%%%%%%%%%%%%%%%%%%%%%%%%%%%

\begin{IEEEbiography}
	[{\includegraphics[width=1.1in,height=1.0in,clip,keepaspectratio]{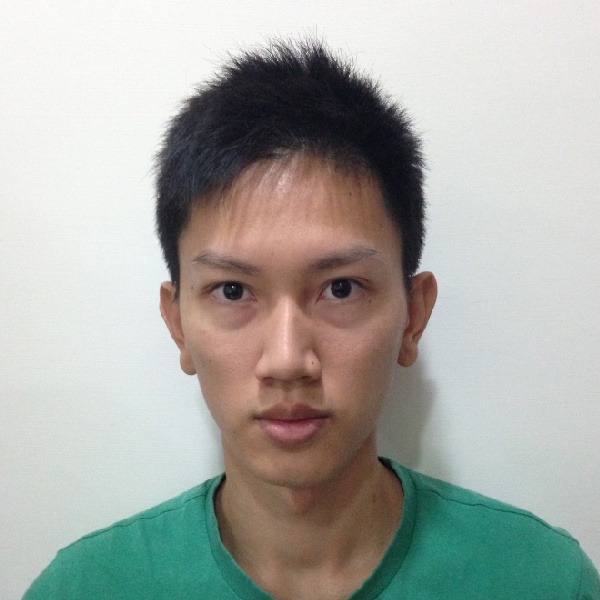}}] {Ting-Jui Chang} received the B.S. degree in electrical and computer engineering from National Chiao Tung University, Taiwan, in 2016, and the M.S. degrees in computer engineering from Texas A\&M University (TAMU), USA, in 2018. He is currently working toward the Ph.D. degree in industrial engineering at Northeastern University. His research interests include distributed learning and optimization, decentralized and online control.
	\end{IEEEbiography}

\begin{IEEEbiography}
	[{\includegraphics[width=1.35in,height=1.1in,clip,keepaspectratio]{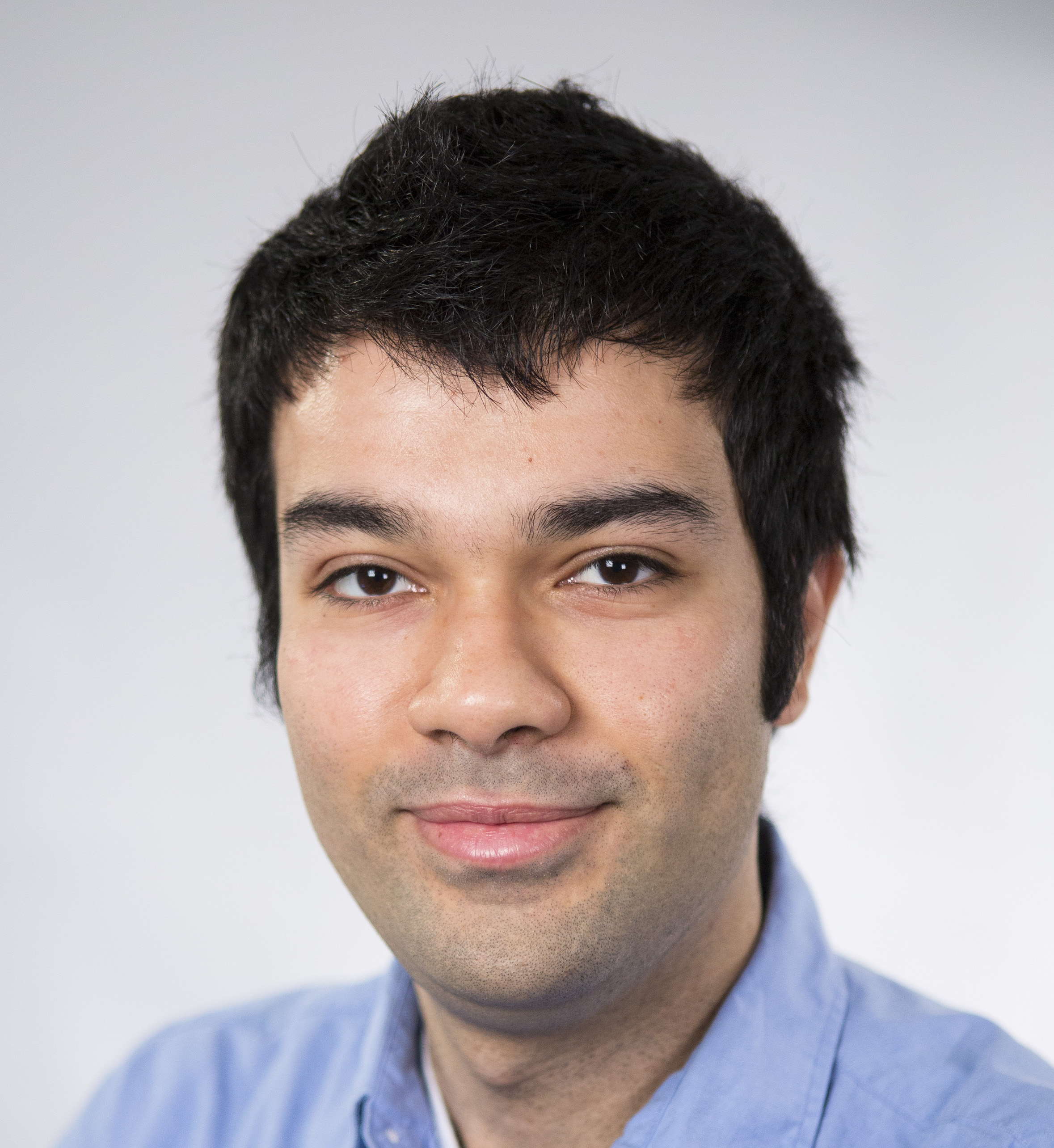}}] {Shahin Shahrampour} received the Ph.D. degree in Electrical and Systems Engineering, the M.A. degree in Statistics (The Wharton School), and the M.S.E. degree in Electrical Engineering, all from the University of Pennsylvania, in 2015, 2014, and 2012, respectively. He is currently an Assistant Professor in the Department of Mechanical and Industrial Engineering at Northeastern University. His research interests include machine learning, optimization, sequential decision-making, and distributed learning, with a focus on developing computationally efficient methods for data analytics. He is a Senior Member of the IEEE.
	\end{IEEEbiography}

\end{document}

%% file: header.tex
\newcommand{\0}{\mathbb{0}}
\newcommand{\1}{\mathbb{1}}

\newcommand{\ab}{\mathbf{a}}

\newcommand{\eb}{\mathbf{e}}

\newcommand{\pb}{\mathbf{p}}

\newcommand{\ub}{\mathbf{u}}
\newcommand{\wb}{\mathbf{w}}
\newcommand{\xb}{\mathbf{x}}
\newcommand{\yb}{\mathbf{y}}
\newcommand{\zb}{\mathbf{z}}

\newcommand{\Ab}{\mathbf{A}}
\newcommand{\Bb}{\mathbf{B}}
\newcommand{\Cb}{\mathbf{C}}
\newcommand{\Db}{\mathbf{D}}
\newcommand{\Eb}{\mathbf{E}}

\newcommand{\Hb}{\mathbf{H}}
\newcommand{\Ib}{\mathbf{I}}

\newcommand{\Kb}{\mathbf{K}}
\newcommand{\Lb}{\mathbf{L}}

\newcommand{\Pb}{\mathbf{P}}
\newcommand{\Qb}{\mathbf{Q}}
\newcommand{\Rb}{\mathbf{R}}

\newcommand{\Vb}{\mathbf{V}}

\newcommand{\Wb}{\mathbf{W}}
\newcommand{\Xb}{\mathbf{X}}

\newcommand{\Dc}{\mathcal{D}}

\newcommand{\Fc}{\mathcal{F}}
\newcommand{\Hc}{\mathcal{H}}

\newcommand{\Nc}{\mathcal{N}}

\newcommand{\Sc}{\mathcal{S}}

\newcommand{\minimize}{\text{minimize}}

\newcommand{\norm}[1]{\lVert#1\rVert}

\newtheorem{theorem}{Theorem}

\newtheorem{corollary}[theorem]{Corollary}
\newtheorem{definition}{Definition}
\newtheorem{example}{Example}
\newtheorem{lemma}[theorem]{Lemma}

\newtheorem{remark}{Remark}